\documentclass[final]{siamltex}
\usepackage{amsmath,amsfonts,amssymb}
\usepackage{latexsym}
\usepackage{epsfig,overpic}
\usepackage{todonotes}
\usepackage{booktabs}
\usepackage{hyperref}
\usepackage{pgfplots}
\usetikzlibrary{shapes,arrows,snakes,calendar,matrix,backgrounds,folding,calc,positioning,spy}
\usepackage{tikz}

\usepackage[scientific-notation=true]{siunitx}
\newcommand{\I}{\mathcal I}

\newcommand{\T}{\mathcal T}
\newcommand{\F}{\mathcal F}

\newcommand{\id}{{\rm id}}

\newcommand{\V}{\mathcal V}
\newcommand{\Vreg}{V_{\text{reg},h}}

\newcommand{\Vregphi}{V_{\text{reg}}}

\newcommand{\rr}{\mathbb{R}}

\newcommand{\jumpleft}{[\![}
\newcommand{\jumpright}{]\!]}
\newcommand{\jump}[1]{\jumpleft #1 \jumpright}
\newcommand{\spacejump}[1]{\jump{#1}}
\newcommand{\averageleft}{\{\!\!\{}
\newcommand{\averageright}{\}\!\!\}}
\newcommand{\average}[1]{\averageleft #1 \averageright}

\renewcommand{\div}{\textrm{div}\ \!}

\newcommand{\hphi}{\hat \phi_h}
\newcommand{\lin}{{\text{\tiny lin}}}
\newcommand{\Omegalin}{\Omega^{ \lin}}
\newcommand{\Gammalin}{\Gamma^{ \lin}}
\newcommand{\OGamma}{\Omega^{\Gamma}}
\newcommand{\OGammaplus}{\Omega^{\Gamma}_+}
\newcommand{\TGamma}{{\mathcal T}^{\Gamma}}

\newcommand{\TGammaplus}{{\mathcal T}^{\Gamma}_+}

\newcommand{\thetah}{\Theta_h}
\newcommand{\thetahGamma}{\Theta_h^\Gamma}
\newcommand{\PsiGamma}{\Psi^\Gamma}

\DeclareGraphicsExtensions{.pdf,.eps,.ps,.eps.gz,.ps.gz,.eps.Y}

\newtheorem{remark}{Remark}

\title{Analysis of a high order unfitted finite element method for elliptic interface problems}
\author{Christoph Lehrenfeld\thanks{Institut f\"ur Numerische und Angewandte Mathematik, University of G\"ottingen, D-37083 G\"ottingen,
Germany; email: {\tt lehrenfeld@math.uni-goettingen.de}}
\and Arnold Reusken\thanks{Institut f\"ur
Geometrie und Praktische  Mathematik, RWTH Aachen University, D-52056 Aachen,
Germany; email: {\tt reusken@igpm.rwth-aachen.de}}
}

\begin{document}
\maketitle
\begin{abstract}
In the context of unfitted finite element discretizations the realization of high order methods is challenging due to the fact that the geometry approximation has to be sufficiently accurate. 
We consider a new  unfitted finite element method which achieves a high order approximation of the geometry for domains which are  implicitly described by smooth level set functions. The method is based on  a parametric mapping which transforms a piecewise planar interface reconstruction to a high order approximation. Both components, the piecewise planar interface reconstruction and the parametric mapping are easy to implement. In this paper we present an  a priori error analysis of the method applied to an interface problem. The analysis reveals optimal order error bounds for the geometry approximation and for the finite element approximation,  for arbitrary high order discretization. The theoretical results are confirmed in numerical experiments.
\end{abstract}


\section[Introduction]{Introduction}\label{sec:introduction}
\subsection{Motivation}
In  recent years there has been a strong increase in the research on  development and analysis of \emph{unfitted finite element methods}. In unfitted finite element methods 
a geometry description which is separated from the computational mesh is used to provide a more flexible handling of the geometry compared to traditional conforming mesh descriptions. 
Often a level set function is used to describe geometries implicitly. 
Recently, significant progress has been made in the construction, analysis and application of methods of this kind, see for instance the papers \cite{burman2014cutfem,burman2012fictitious,fries2010extended,grossreusken07,hansbo2002unfitted,olshanskii2009finite}. Despite these  achievements, the development and rigorous error analysis of \emph{high order} accurate unfitted finite element methods is still challenging.
This is mainly due to the fact that efficient, highly accurate numerical integration on domains that are implicitly described is not straight-forward.
In the recent paper \cite{lehrenfeld15}, a new approach based on isoparametric mappings of the underlying mesh (outlined in section~\ref{secm} below) has been introduced which allows for an efficient, robust and highly accurate numerical integration on domains that are described  implicitly by a level set function. The main contribution of this paper is an error analysis of this method applied to a model interface problem.

\subsection{Literature}
We review the state of the art with respect to higher order geometry approximations in the literature to put the approach proposed in \cite{lehrenfeld15} and analysed in this paper into its context.

For discretizations based on piecewise linear (unfitted) finite elements a numerical integration approach which is second order accurate suffices to preserve the overall order of accuracy. This is why an approximation with piecewise planar geometries is an established approach. Piecewise planar approximations allow for a tesselation into simpler geometries, e.g. simplices, on which standard quadrature rules are applicable. Hence, a robust realization is fairly simple. It can be applied to quadrilateral and hexahedral meshes, as in the marching cube algorithm \cite{lorensen1987marching} as well as to simplex meshes, cf. (among others) \cite[Chapter 4]{lehrenfeld2015diss}, \cite{mayer2009interface}, \cite[Chapter 5]{naerland2014geometrychap5} for triangles and tetrahedra and  \cite{lehrenfeld2015nitsche},\cite[Chapter 4]{lehrenfeld2015diss} for 4-prisms and pentatopes (4-simplices). Such strategies are used in many simulation codes, e.g. \cite{burman2014cutfem,carraro2015implementation,engwer2012dune,DROPS,renard2014getfem++} and are often combined with a geometrical refinement in the quadrature. 
Especially on octree-based meshes this can be done very efficiently \cite{chernyshenko2015adaptive}.
However, by construction, the tesselation approach is only second order accurate and hence limits the overall accuracy when combined with higher order finite elements.

Many unfitted discretizations based on piecewise linear finite elements have a natural extension to higher order finite element spaces, see for instance \cite{bastian2009unfitted,johanssonhigh2013,massjung12,parvizianduesterrank07}. 
It requires, however, additional new techniques to obtain also higher order accuracy when errors due to the geometry approximation and quadrature are taken into account. We review such techniques, starting with approaches from the engineering literature.

One such a technique is based on  moment fitting approaches, resulting in special quadrature rules  \cite{muller2013highly,sudhakar2013quadrature}.
This approach provides (arbitrary) high order accurate integration on implicit domains. However, the construction of these quadrature rules is fairly involved and the positiveness of quadrature weights can not be guaranteed which can lead to stability problems. Furthermore, as far as we know there is no rigorous complexity and error analysis.
For special cases also other techniques have been proposed in the literature to obtain high order accurate approximations of integrals on unfitted domains.
In \cite{saye2015hoquad} a new algorithm is presented which can achieve arbitrary high order accuracy and guarantee positivity of integration weights. However, it  is applicable only on tensor-product elements. The approach is based on the idea of locally interpreting the interface as a graph over a hyperplane. 

In the community of the extended finite element method (XFEM) approaches applying a parametric mapping of the sub-trianguation are often used, e.g., \cite{cheng2010higher,dreau2010studied,fries2015}.  The realization of such strategies is technically involved, especially in three (or higher) dimensions. Moreover, ensuring robustness  (theoretically and practically) of these approaches is difficult. 
Concerning the finite cell method (FCM) \cite{parvizianduesterrank07}, which is also a (high order) unfitted finite element method, we refer to \cite{abedian2013performance} for a comparison of different approaches for the numerical integration on unfitted geometries.

We note that  the aforementioned papers on numerical integration in higher order unfitted finite element methods, which are all from the engineering literature, only present new approaches for the numerical integration and carry out numerical convergence studies. Rigorous error analyses of these methods are not known. 

We now discuss techniques  and corresponding rigorous error analyses  that have been considered in the mathematical literature.
In the recent papers \cite{boiveau2016fictitious,burman2015cut} the piecewise planar approximation discussed before has been used and combined with a correction in the imposition of boundary values (based on Taylor expansions) to allow for higher order accuracy in fictitious domain methods applied to Dirichlet problems. Optimal order a priori error bounds are derived.
For a higher order unfitted discretization of partial differential equations on surfaces, a parametric mapping of a piecewise planar interface approximation has been developed and analyzed in \cite{grande2014highorder}. These are the only papers, that we know of, in which geometry errors in higher order unfitted finite element methods are included in the error analysis.

We also mention relevant numerical analysis papers using low order unfitted or higher order geometrically conforming (``fitted'') methods.
For unfitted piecewise linear  finite element discretizations of partial differential equations on surfaces, errors caused  by geometry approximation are analyzed in \cite{deckelnick2014unfitted,cutDG}.
Error analyses for fitted isoparametric finite elements for the accurate approximation of curved boundaries have been developed in the classical contributions 
\cite{bernardi1989optimal,CiarletHandbook,lenoir1986optimal}. In the context of interface problems and surface-bulk coupling isoparametric fitted techniques are analyzed in  \cite{elliott2012finite}.

Recently, these isoparametric fitted methods have also been used
to adjust a simple background mesh to curved implicitly defined geometries in a higher order accurate fashion \cite{Basting2013228,demlow2009higher,gawlik2014high,omerovicconformal2016}.

The approach analyzed in this paper is similar to above-mentioned approaches \cite{Basting2013228,cheng2010higher,dreau2010studied,gawlik2014high,grande2014highorder,lenoir1986optimal,omerovicconformal2016} in that it is also based on a piecewise planar geometry approximation which is significantly improved using a parametric mapping. 
The important difference compared to \cite{Basting2013228,gawlik2014high,lenoir1986optimal,omerovicconformal2016} is that we consider an \emph{un}fitted discretization and compared to \cite{cheng2010higher,dreau2010studied,grande2014highorder} is that we consider a \emph{parametric mapping of the underlying mesh} rather than the sub-triangulation or only the interface.

The discretization on the higher order geometry approximation that we obtain from a parametric mapping is based on a variant of Nitsche's method \cite{Nitsche71} for the imposition of interface conditions in non-conforming unfitted finite element spaces \cite{hansbo2002unfitted}. To our knowledge, there is no literature in which geometrical errors for isoparamteric or higher order unfitted Nitsche-type discretizations has been considered.

\subsection{The problem setting} 
On a bounded connected polygonal domain $\Omega \subset \Bbb{R}^d$, $d=2,3$, we consider the model interface problem
\begin{subequations} \label{eq:ellmodel}
\begin{align}
- \mathrm{div} (\alpha_i \nabla {u}) &= \, f_i 
\quad \text{in}~~ \Omega_i , ~i=1,2, \label{eq:ellmodel1} \\
\spacejump{{\alpha} \nabla {u} }_{\Gamma} \cdot n_\Gamma &= \, 0, \quad \spacejump{{u}}_{\Gamma} = 0 \quad \text{on}~~\Gamma, \label{eq:ellmodel2} \\
u &= 0 \quad \text{ on } \partial \Omega.
 \end{align}
\end{subequations}
Here, $\Omega_1 \cup \Omega_2= \Omega $ is a nonoverlapping partitioning of the domain, $\Gamma = \bar \Omega_1 \cap \bar \Omega_2$ is the interface and  $\spacejump{\cdot}_{\Gamma}$ denotes the usual jump operator across  $\Gamma$. 
$f_i$, $i=1,2$ are domain-wise described sources. In the remainder we will also use the source term $f$ on $\Omega$ which we define as $f|_{\Omega_i} = f_i$, $i=1,2$.
The diffusion coefficient $\alpha$ is assumed to be piecewise constant, i.e. it has a constant value $\alpha_i$  on each sub-domain $\Omega_i$.
 We assume simplicial triangulations of $\Omega$ which are \emph{not} fitted to $\Gamma$. 
Furthermore, the interface is characterized as the zero level of a given level set function $\phi$.

\subsection{Basic idea of the isoparametric unfitted discretization} \label{secm}
The new idea of the method introduced in \cite{lehrenfeld15} is to construct a parametric mapping of the underlying triangulation based on a \emph{higher order} (i.e. degree at least 2) finite element approximation $\phi_h$ of the level set function $\phi$ which characterizes the interface. This mapping, which is easy to construct, uses  information extracted from the level set function  to map a \emph{piecewise planar} interface approximation of the interface to a \emph{higher order} approximation, cf. the sketch in Figure \ref{fig:idea}.
\begin{figure}[h!]
  \small
  \begin{center}
    \hspace*{-0.25cm}
    \begin{tabular}{c@{}c@{}c@{}c@{}c}
        \begin{minipage}{0.27\textwidth}
          \includegraphics[width=\textwidth]{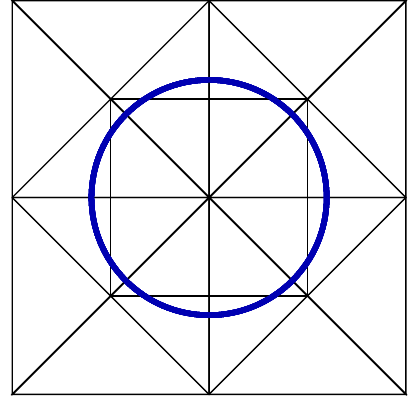} 
        \end{minipage}
      &
        \begin{minipage}{0.038\textwidth}
          \centering +
        \end{minipage}
      &
        \begin{minipage}{0.27\textwidth}
          \includegraphics[width=\textwidth]{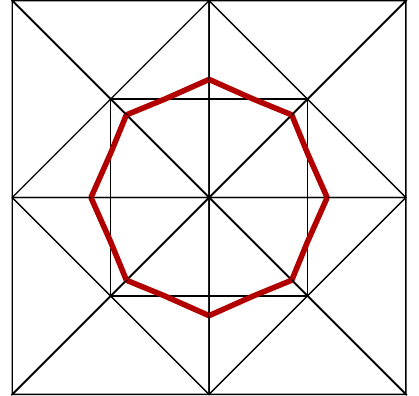} 
        \end{minipage}
      &
        \begin{minipage}{0.052\textwidth}
          \large
          \centering
          $
          \displaystyle
          \stackrel{\Theta_h}{\longrightarrow}
          $
          \small
        \end{minipage}
      &
        \begin{minipage}{0.27\textwidth}
          \includegraphics[width=\textwidth]{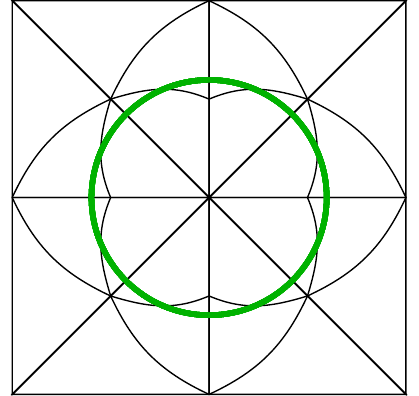} 
        \end{minipage}
    \end{tabular}
  \end{center}
  \caption{Basic idea of the method in \cite{lehrenfeld15}: The geometry description with the level set function approximation $\phi_h$ is highly accurate but implicit (left). The zero level  $\Gamma^{\text{lin}}$ of the piecewise linear interpolation $\hat{\phi}_h$ has an  explicit representation but is only second order accurate (center). $\Gamma^{\text{lin}}$ is mapped approximately to the implicit interface $\{\phi_h=0\}$ using a  mesh transformation $\Theta_h$, resulting in a highly accurate and explicit representation $\Gamma_h = \Gamma^{\text{lin}} \circ \Theta_h^{-1}$ (right).}
  \vspace*{-0.2cm}
  \label{fig:idea} 
\end{figure}

A higher order isoparametric unfitted finite element method is then obtained by first formulating a higher order discretization with respect to the low order geometry approximation which serves as a reference configuration. The parametric mapping is then applied to improve the accuracy of the geometry approximation. The application of the transformation also induces a change of the finite element space which renders the space an isoparametric finite element space.  
This approach is geometry-based and can be applied to unfitted interface or boundary value problems as well as to partial differential equations on surfaces.  
All volume and interface integrals that occur in the implementation of the method can be formulated in terms of integrals on the reference configuration, i.e. on simplices cut by the piecewise \emph{planar} approximation of the interface. Hence, quadrature is straightforward. 

For the definition of the discretization of the model problem \eqref{eq:ellmodel} on the reference geometry we take an \emph{un}fitted higher order ``cut'' finite element space, where the cut occurs at the piecewise \emph{planar} zero level of the piecewise linear finite element approximation of $\phi_h$. 
Applying the mesh transformation to this space induces an isoparametric unfitted finite element space for the discretization of \eqref{eq:ellmodel}. In the same way as in the seminal paper \cite{hansbo2002unfitted} the continuity of the discrete solution across the (numerical) interface is enforced in a weak sense using Nitsche's method. 
The resulting discrete problem has a unique solution $u_h$ in the isoparametric unfitted finite element space  ($h$ is related in the usual way to the size of the simplices in the triangulation). 

\subsection{Content and structure of the paper}
In this paper we present a rigorous error analysis of the method introduced in \cite{lehrenfeld15} applied to the interface problem \eqref{eq:ellmodel}. 
The main result of the paper is the discretization error bound in Theorem~\ref{mainthm}. 
As a corollary of that result we obtain for isoparametric unfitted finite elements of degree $k \geq 1$ an  error bound of the form
 \[
   h^{-\frac12} \Vert \spacejump{ \tilde{u}_h}_\Gamma \Vert_{L^2(\Gamma)} + 
   \|u- \tilde{u}_h \|_{H^1(\Omega_1 \cup \Omega_2)} \leq c h^k\big( \|u\|_{H^{k+1}(\Omega_1 \cup \Omega_2)}  + \|f\|_{H^{1,\infty}(\Omega_1 \cup \Omega_2)} \big),
\]
with a constant $c>0$ that is independent of $h$ and the interface position within the computational mesh. $\tilde{u}_h := u_h \circ \Phi_h^{-1}$ is the mapped discrete solution where 
$\Phi_h$ is a smooth transformation which maps from the approximated to the exact domains $\Omega_i,~i=1,2$ and is close to the identity (precise definition given in section~\ref{sectmappsi}).
Here, $\|\cdot\|_{H^1(\Omega_1 \cup \Omega_2)}=\|\cdot\|_{H^{1,2}(\Omega_1 \cup \Omega_2)}$ and  $\|\cdot\|_{H^{1,\infty}(\Omega_1 \cup \Omega_2)}$ are broken Sobolev norms 
with 
$
\|\cdot\|_{W(\Omega_1 \cup \Omega_2)}^2 
= \sum_{i=1,2} \|\cdot\|_{W(\Omega_i)}^2 
$, $W \in \{ H^1, H^{k+1}, H^{1,\infty} \}$.
As far as we know this is the first rigorous \emph{higher order} error bound for an \emph{un}fitted finite element method applied to a problem with an \emph{implicitly} given interface.

The paper is organized as follows. In Section \ref{prelim} we introduce notation and assumptions and define a finite element projection operator. The parametric mapping $\Theta_h$ is introduced in Section \ref{sectmeshtransform}. In Section \ref{unfittedFEM} the isoparametric unfitted finite element method is presented and numerical examples are given. The main contribution of this paper is the error analysis of the  method given in Section \ref{sec:erroranalysis}.
\newpage

\section{Preliminaries} \label{prelim}
\subsection{Notation and assumptions} \label{notassumption}
We introduce notation and assumptions. 
The simplicial triangulation of $\Omega$ is denoted by $\T$ and the standard  finite element space of continuous piecewise polynomials up to degree $k$ by $V_h^k$. 
The nodal interpolation operator in $V_h^k$ is denoted by $I_k$. 

For ease of presentation we assume quasi-uniformity of the mesh, s.t. $h$ denotes a characteristic mesh size with $h\sim h_T:={\rm diam}(T)$, $~T\in\T$.

We assume that the smooth interface $\Gamma$ is the zero level of a smooth level set function $\phi$, i.e.,  $\Gamma= \{\, x \in \Omega~|~\phi(x)=0\,\}$.
This level set function is not necessarily close to a distance function, but has the usual properties of a level set function: 
\begin{equation} \label{LSdef}
\|\nabla \phi(x)\| \sim 1~,~~\|D^2 \phi(x)\| \leq c \quad \text{for all}~x~~ \text{in a neighborhood $U$ of $\Gamma$}.
\end{equation}
We assume that the level set function has the smoothness property  $\phi \in C^{k+2}(U)$. The assumptions on the level set function \eqref{LSdef} imply the following relation, which is fundamental in the analysis below
\begin{equation} \label{eq1}
 |\phi(x+\epsilon \nabla \phi(x)) - \phi(x+\tilde \epsilon \nabla \phi(x))| \sim |\epsilon - \tilde \epsilon| \quad x \in U,
\end{equation}
for $|\epsilon|, |\tilde \epsilon|$ sufficiently small.

As input for the parametric mapping we need an approximation $\phi_h \in V_h^k$ of $\phi$, and we assume that this approximation satisfies the error estimate
\begin{equation} \label{err2}
  \max_{T\in \T} |\phi_h - \phi|_{m,\infty,T \cap U} \lesssim h^{k+1-m},\quad 0 \leq m \leq k+1.
\end{equation}
Here $|\cdot|_{m,\infty,T\cap U}$ denotes the usual semi-norm on the Sobolev space $H^{m,\infty}(T\cap U)$.
Note that  \eqref{err2} holds for the nodal interpolation $\phi_h = I_k \phi$ and implies the estimate
\begin{equation} \label{err1}
 \|\phi_h - \phi\|_{\infty, U} + h\|\nabla(\phi_h - \phi)\|_{\infty, U} \lesssim h^{k+1}. 
\end{equation}
Here and in the remainder we use the notation $\lesssim$ (and $\sim$ for $\lesssim$ and $\gtrsim$), which denotes an  inequality with a constant that is independent of $h$ and of how the interface $\Gamma$ intersects the triangulation $\T$. This  constant may depend on $\phi$ and on the diffusion coefficient $\alpha$, cf.~\eqref{eq:ellmodel1}. In particular, the estimates that we derive are not uniform in the jump in the diffusion coefficient $\alpha_1/\alpha_2$.

The zero level of the finite element function $\phi_h$ (implicitly) characterizes the discrete interface. 
\emph{The piecewise linear nodal interpolation of $\phi_h$ is denoted by $\hphi = I_1 \phi_h$.} Hence, $\hphi(x_i)=\phi_h(x_i)$ at all vertices $x_i$ in the triangulation $\T$. 
The low order geometry approximation of the interface, which is needed in our discretization  method, is the zero level of this function, $\Gammalin := \{ \hphi = 0\}$. 
The corresponding sub-domains are denoted as $\Omegalin_i = \{ \hphi \lessgtr 0 \}$. 
All elements in the triangulation $\T$ which are cut by $\Gammalin$ are collected in the set $\TGamma := \{T \in \T, T \cap \Gammalin \neq \emptyset \}$. The corresponding domain is $\OGamma := \{ x \in T, T\in \TGamma\}$. The extended set which includes all neighbors that share at least one vertex with elements in $\TGamma$ is $\TGammaplus := \{T \in \T, T\cap \OGamma \neq \emptyset \}$ with the corresponding domain $\OGammaplus := \{ x \in T, T \in \TGamma\}$.

\subsection{Projection operator onto the finite element space $V_h^k(\Omega^\Gamma)^d$} \label{sec:proj}
In the construction of the isoparametric mapping $\Theta_h$ we need a projection step from a function which is piecewise polynomial but discontinuous (across element interfaces) to the space of continuous finite element functions. 
Let $C(\TGamma) = \bigoplus\limits_{T\in \TGamma} C(T)$ and $V_h^k(\OGamma) := V_h^k|_{\OGamma}$. We introduce a projection operator $P_h^\Gamma: C(\TGamma)^d \rightarrow V_h^k(\OGamma)^d$. 
The projection operator relies on a nodal representation of the finite element space $V_h^k(\OGamma)$.
The set of finite element nodes $x_i$ in $\TGamma$ is denoted by $N(\TGamma)$, and $N(T)$ denotes the set of finite element nodes associated to $T \in \TGamma $. All elements $T \in \TGamma$ which contain the same finite element node $x_i$ form the set denoted by $\omega(x_i)$: 
\begin{align*}
\omega(x_i) & := \{ T \in \TGamma | x_i \in N(T) \}, ~ x_i \in N(\TGamma).
\end{align*}
For each finite element node we define the local average as 
\begin{equation} \label{average}
  A_{x_i}(v):= \frac{1}{|\omega(x_i)|} \sum_{T \in \omega(x_i)} v_{|T}(x_i), ~~x_i \in N(\TGamma).
\end{equation}
where $|\cdot|$ denotes the cardinality of the set $\omega(x_i)$.
The projection operator $P_h^\Gamma: C(\TGamma)^d \to V_h^k(\OGamma)^d$ is defined as
\begin{equation*} 
P_h^\Gamma v:= \sum_{x_i \in N(\TGamma)} A_{x_i}(v) \, \psi_i, \quad v \in C(\TGamma)^d,
\end{equation*}
where $\psi_i$ is the nodal basis function corresponding to $x_i$. This is a simple and well-known projection operator considered also in e.g., \cite[Eqs.(25)-(26)]{oswald} and \cite{ernguermond15}.
Note that 
\begin{equation} P_h^\Gamma w = I_k w ~~\forall~ w \in C(\OGamma)^d, \quad  \|P_h^\Gamma w\|_{\infty,\Omega} \lesssim \max_{T \in\TGamma} \| w\|_{\infty,T}~~\forall~ w \in C(\TGamma)^d.
\end{equation} 

\section{The isoparametric mapping $\Theta_h$} \label{sectmeshtransform}
In this section we introduce the transformation $\Theta_h \in (V_h^k)^d$ which is a bijection on $\Omega$ and satisfies $\Theta_h = \id$ on $\Omega \setminus \OGammaplus$. This mapping is constructed in two steps. First a local mapping $\Theta_h^\Gamma$, which is defined on $\OGamma$, is derived and this local mapping is then extended to the whole domain. We also need another bijective mapping on $\Omega$, denoted by $\Psi$, which is constructed in a similar two-step procedure. This mapping $\Psi$ is needed in the error analysis and in the derivation of important properties of $\Theta_h$.  In the higher order finite element method, which is presented in Section~\ref{unfittedFEM},  the mapping $\Theta_h$ is a key component, and the mapping $\Psi$ is \emph{not} used.  In this section we introduce both $\Theta_h$ and $\Psi$, because the construction of both mappings has strong similarities. Since the construction is not standard and consists of a two-step procedure, we outline our approach, cf. Fig.~\ref{fig:trafos}:

\begin{itemize}
 \item We start with the relatively simple definition of the local mapping $\Psi^\Gamma$, which has the property $\Psi^\Gamma(\Gammalin)=\Gamma$ (Section~\ref{subsctionconstr0}), but is only defined in $\Omega^\Gamma$. In general this mapping can not (efficiently) be constructed in practice.
 \item The definition of $\Psi^\Gamma$ is slightly modified, to allow for a computationally efficient construction, which then results in the local isoparametric mapping $\thetahGamma$ with the property $\thetahGamma(\Gammalin)=\Gamma_h \approx \Gamma$ (Section~\ref{subsctionconstr1}). $\thetahGamma$ is only defined in $\Omega^\Gamma$.
 \item The extension of the mappings $\Psi^\Gamma$ and $\thetahGamma$ is based on the same general procedure, which is derived  from extension techniques that are standard in the literature on isoparametric finite element methods, cf. \cite{lenoir1986optimal, bernardi1989optimal}. This general procedure is explained in Section~\ref{sectextension}.
 \item The general extension procedure is applied to the mapping $\Psi^\Gamma$, resulting in the global bijection $\Psi$. Important properties of $\Psi$ are derived (Section~\ref{sec:globaltrafoPsi}).
 \item Finally the general extension procedure is applied to the mapping $\thetahGamma$, resulting in the global bijection $\thetah$. Important properties of $\thetah$ are derived (Section~\ref{sec:globaltrafo}).
\end{itemize}

\begin{figure}[h!]
  \vspace*{-0.3cm}
  \begin{center}
    \includegraphics[width=0.86\textwidth]{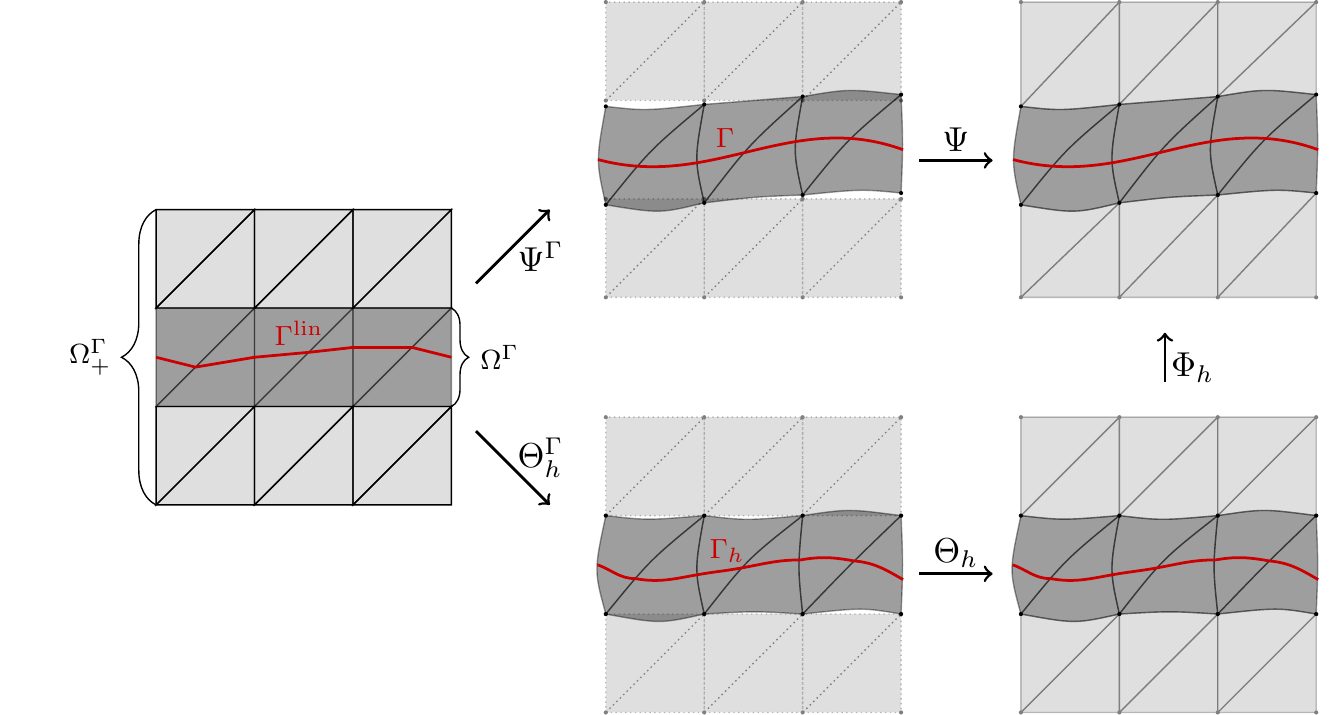}
  \end{center}
  \vspace*{-0.2cm}
   \caption{Sketch of different mesh transformations $\Psi^\Gamma$, $\thetahGamma$, $\Psi$, $\Theta_h$ and $\Phi_h := \Psi \circ \Theta_h^{-1}$.}
   \label{fig:trafos}
 \end{figure}

\subsection{Construction of $\Psi^\Gamma$}\label{subsctionconstr0}

 We introduce  the search direction $G:=\nabla \phi$ 
and a  function $d: \OGamma \to \rr$ defined as follows:  $d(x)$ is the (in absolute value) smallest number such that
\begin{equation} \label{cond1}
  \phi(x + d(x) G(x))=\hat \phi_h(x)  \quad \text{for}~~x \in \OGamma.
\end{equation}
(Recall that $\hat\phi_h$ is the piecewise linear nodal interpolation of $\phi_h$.)
We summarize a few properties of the function $d$.
\begin{lemma} \label{lem6}
For $h$ sufficiently small, the relation \eqref{cond1} defines a unique $d(x)$ and $d \in C(\OGamma) \cap H^{1,\infty}(\OGamma) \cap C^{k+1}(\TGamma)$. Furthermore there holds
\begin{subequations}
\begin{align}
        |d(x_i)| & \lesssim h^{k+1} \quad \text{for all {\rm vertices} $x_i$ of}~ T \in \TGamma ,\label{resd4a0}\\
       \|d\|_{\infty,\OGamma} &\lesssim h^2, \label{resd4a} \\
         \|d\|_{H^{1,\infty}(\OGamma)} &\lesssim h, \label{resd4b} \\
      \max_{T \in \TGamma} \|d\|_{H^{l,\infty}(T)} &\lesssim 1,~ ~\text{for}~~ l \leq k+1. \label{resd4c}
\end{align}
\end{subequations}
\end{lemma}
\begin{proof} 
For $|\alpha| \leq \alpha_0 h$, with a fixed $\alpha_0 >0 $, we introduce for a fixed $x \in \OGamma$  the continuous function
\[
 g(\alpha):=\phi(x+ \alpha G(x))- \hat \phi_h(x).
\] 
From \eqref{err1} one obtains $\|\hat \phi_h - \phi\|_{\infty,\OGamma} \lesssim h^2$. Using this, a triangle inequality and \eqref{LSdef} we get
\begin{equation} \label{lpp}
 g(\alpha)= \alpha \|\nabla \phi(x)\|_2^2 + \mathcal{O}(h^2).
\end{equation}
Hence, there exists $h_0 >0$ such that for all $0< h\leq h_0$ and all $x\in \OGamma$ the equation $g(\alpha)=0$ has a unique solution $\alpha =:d(x)$ in $[-\alpha_0 h,\alpha_0 h]$. The result in \eqref{resd4a} follows from \eqref{lpp}. Using  \eqref{eq1} and \eqref{err1} we get, for a vertex $x_i$ of $ T \in \TGamma$,
\[
 |d(x_i)| \sim |\phi(x_i+d(x_i)G(x_i))-\phi(x_i)|= |\hat\phi_h(x_i)-\phi(x_i)|=|\phi_h(x_i)-\phi(x_i)| \lesssim h^{k+1},
\]
which is the estimate in \eqref{resd4a0}. The continuity of $d$ on $\OGamma$ follows from the continuity of $\phi, G$ and $\hat \phi$.

For \eqref{resd4b} we differentiate \eqref{cond1} and write $y = x + d(x) G(x)$:
\begin{align*}
  \nabla \phi(y) + \nabla d \nabla \phi(y)^T G + d (D G^T)  \nabla \phi(y) & = \nabla \hphi(x).
\end{align*}
This yields
\begin{align*}
  \nabla d ( \Vert \nabla \phi \Vert_2^2 + \mathcal{O}(h^2)) + \mathcal{O}(h^2) & = \nabla \hphi(x) - \nabla \phi(y) 
 = \nabla \hphi(x) - \nabla \phi(x) +  \mathcal{O}(h^2),
\end{align*}
which, in combination with $\|\nabla \hphi - \nabla \phi\|_{\infty,\OGamma} \lesssim h$,  implies $\Vert \nabla d(x)  \Vert \lesssim h$ a.e. in $\OGamma$. 

For $T\in \TGamma$ we consider the function $F(x,y) = \phi(x+y G(x))-\hphi(x),~(x,y)\in T \times (-\alpha_0 h,\alpha_0 h)$. The function $d(x)=y(x)$ solves the implicit equation $F(x,y(x))=0$ on $T$. We recall the regularity assumption $\phi\in C^{k+2}(U)$. From the implicit function theorem we know that $y \in C^{k+1}(T)$, because $F \in C^{k+1}\big(T \times (- \alpha_0 h, \alpha_0 h)\big)$.  From $y \in C^{k+1}(\TGamma)$ and $y \in C(\OGamma)$ we conclude $y \in H^{1,\infty}(\OGamma)$. Note that $D^\alpha \hat \phi_h=0 $ for $|\alpha| \geq 2$ and $\|D^\alpha \hat \phi_h\|_{\infty,T} \lesssim \| \phi \|_{H^{2,\infty}(T)}$ for $|\alpha| \leq 1$. Hence,
\begin{equation} \label{lp}
 \|D^\alpha_{(x,y)}F\|_{\infty, T \times (\-\alpha_0h, \alpha_0 h)} \lesssim \|\phi\|_{H^{l+1,\infty}(U)}\quad \text{for}~|\alpha| \leq k+1.
\end{equation}
 Differentiating $F(x,y(x))=0$ yields
\begin{equation}\label{eq:impl:1}
  D^\alpha y(x) = -D_y F(x,y(x))^{-1} D_x^{\alpha} F(x,y(x)) = -A(x) D_x^{\alpha} F(x,y(x)),  \quad |\alpha|=1,
\end{equation}
with $A(x) = S(x)^{-1}$, $S(x) = D_y F(x,y(x))= \nabla \phi(x+yG(x))^T\nabla \phi(x) \in [c_0,c_1]$ with $c_0, c_1 > 0$ independent of $h,x,T$.   Differentiating $S(x) A(x) = I$ we get 
\begin{equation}\label{eq:impl:2}
  D^\alpha A(x) = - A(x)^2 D^\alpha S(x), \quad |\alpha|=1. 
\end{equation}
From the relations \eqref{eq:impl:1} and \eqref{eq:impl:2} it follows  that $|D^\alpha y(x)|$, $|\alpha| = l$ can be bounded in terms of $\|A(x)\|$ and $|D^\alpha_{(x,y)} F(x,y(x))|$, $|\alpha|\leq l$. Combining this with \eqref{lp} proves \eqref{resd4c}.
\end{proof}
\\[1ex]
Given the function $dG \in C(\OGamma)^d \cap H^{1,\infty}(\OGamma)^d$ we define:
\begin{equation} \label{psi1}
 \PsiGamma(x):= x + d(x) G(x), \quad x \in \OGamma.
\end{equation}
Note that the function $d$ and mapping $\Psi^\Gamma$ depend on $h$, through $\hat \phi_h$ in \eqref{cond1}. We do not show this dependence in our notation.
As a direct consequence of the estimates derived in Lemma~\ref{lem6} and the smoothness assumption $G=\nabla \phi \in C^{k+1}(U)$ we have the following uniform bounds on (higher) derivatives of $\Psi^\Gamma$.
\begin{corollary} \label{lem:boundpsig}
  The following holds:
  \begin{align}
 \|\PsiGamma - \id \|_{\infty,\OGamma} +h \|D\PsiGamma - I\|_{\infty,\OGamma} & \lesssim h^2, \label{fistderPsi}\\
   \max_{T \in \TGamma} \Vert D^l \PsiGamma \Vert_{\infty,T} & \lesssim 1, \quad l \leq  k+1.\label{fistderPsiA}
  \end{align}
\end{corollary}

\subsection{Construction of $\thetahGamma$} \label{subsctionconstr1}
The construction of $\thetahGamma$ consists of two steps. In the first step we introduce a discrete analogon of $\Psi^\Gamma$ defined in \eqref{psi1}, denoted by $\Psi_h^\Gamma$. 
Based on this $\Psi_h^\Gamma$, which can be discontinuous across element interfaces,  we obtain a continuous transformation $\thetahGamma \in C(\OGamma)^d$ by averaging with the projection operator $P_h^\Gamma$ from Section \ref{sec:proj}.

For the construction of $\Psi_h^\Gamma$ we need an efficiently  computable good approximation of $G=\nabla \phi$ on $\TGamma$.
For this we consider the following two options
  \begin{equation} \label{eq:gh}
    G_h(x) = \nabla \phi_h(x), \quad \text{or}~~G_h(x) = (P_h^\Gamma \nabla \phi_h)(x), \quad x \in T \in \TGamma.
  \end{equation}
\begin{lemma} \label{lem2} Both for  $G_h = P_h^\Gamma \nabla \phi_h$ and $G_h= \nabla \phi_h$ the estimate
\begin{equation}\label{eqlem2}
  \|G_h - G\|_{\infty, \OGamma} + h \max_{T\in \TGamma} \| D (G_h - G) \|_{\infty, T} \lesssim h^{k}
\end{equation}
holds with $G = \nabla \phi$.
\end{lemma}
\begin{proof}
 For $G_h= \nabla \phi_h$ the result is a direct consequence of  \eqref{err1}. For  $G_h = P_h^\Gamma \nabla \phi_h$ we use 
\begin{align*}
\Vert P_h^\Gamma \nabla \phi_h - \nabla \phi \Vert_{\infty,\OGamma} \leq \Vert P_h^\Gamma \nabla \phi_h - I_k \nabla \phi \Vert_{\infty,\OGamma} + \Vert \nabla \phi - I_k \nabla \phi \Vert_{\infty,\OGamma}, \\
  \Vert P_h^\Gamma \nabla \phi_h - I_k \nabla \phi \Vert_{\infty,\OGamma} = \Vert P_h^\Gamma (\nabla \phi_h - \nabla \phi) \Vert_{\infty,\OGamma} \lesssim \max_{T\in\TGamma} \Vert \nabla \phi_h - \nabla \phi \Vert_{\infty,T},
\end{align*}
and \eqref{err1} to estimate $  \|G_h - G\|_{\infty, \OGamma} \lesssim h^k$. For the derivative we  make use of \eqref{err2} and \eqref{err1}:
\begin{align*}
 \Vert D (G_h-G) \Vert_{\infty,T} 
& \leq \Vert D(G_h - I_k G) \Vert_{\infty,T} +  \Vert D(G - I_k G) \Vert_{\infty,T} \\
& \lesssim \frac{1}{h} \Vert G_h - I_k G \Vert_{\infty,T} +  h^{k-1} \\
& \lesssim \frac{1}{h} \Vert G_h - G \Vert_{\infty,T} + \frac{1}{h} \Vert G - I_k G \Vert_{\infty,T} + h^{k-1} \lesssim h^{k-1},
\end{align*}
with estimates that are uniform in $T \in \TGamma$. 
\end{proof}
\ \\
\begin{remark} \rm We comment on the options in \eqref{eq:gh}.
An important property of the search direction $G_h$ is the proximity to a \emph{continuous} vector field $G$ which describes the direction with respect to which the interface $\Gamma$ can be interpreted as a graph on $\Gammalin$: For each $x \in \Gamma$ there exists a unique $y \in \Gammalin$ and $d\in \rr$, s.t. $x = y + d\, G(x)$.
The choices in \eqref{eq:gh} are accurate approximations of $G(x) = \nabla \phi$ which (locally) have the graph property for sufficiently smooth interfaces $\Gamma$.
Despite the discontinuities across element interfaces, the discrete search direction $G_h = \nabla \phi_h$ is a reasonable choice as the distance to the continuous search direction $G=\nabla \phi$ is sufficiently small, cf. Lemma \ref{lem2}.
\end{remark}\\[1ex]
 Let $\mathcal{E}_T \phi_h$ be the polynomial extension of $\phi_h|_T$. 
We define  a function $d_h: \TGamma \to [-\delta,\delta]
$, with $\delta > 0$ sufficiently small, as follows: $d_h(x)$ is the (in absolute value) smallest number such that  
  \begin{equation} \label{eq:psihmap}
    \mathcal{E}_T \phi_h(x + d_h(x) G_h(x)) = \hat \phi_h, \quad \text{for}~~ x\in  T \in \TGamma.
  \end{equation}
Clearly, this $d_h(x)$ is a ``reasonable'' approximation of the steplength $d(x)$ defined in \eqref{cond1}.
We summarize a few properties of the function $d_h$. 
\begin{lemma} \label{propertiesdh}
 For $h$ sufficiently small, the relation \eqref{eq:psihmap} defines a unique $d_h(x)$ and $d_h \in C^\infty(\TGamma)$. Furthermore:
\begin{subequations}
\begin{align}
 d_h(x_i) & =0 \quad \text{for all {\rm vertices} $x_i$ of $T \in \TGamma$}, \label{resd1} \\
 \max_{T \in \TGamma} \|d_h\|_{\infty,T} & \lesssim h^2, \label{resd4} \\
  \max_{T \in \TGamma} \|\nabla d_h\|_{\infty,T} & \lesssim 1, \label{resd6}
\end{align}
\end{subequations}
 \end{lemma}
\begin{proof}
These results can be derived using arguments very similar to the ones used in the proof of Lemma~\ref{lem6}. 
For completeness such a proof is given in the Appendix, section \ref{proof:propertiesdh}.
\end{proof}
\ \\[1ex]
Given the function $d_h G_h \in  C(\TGamma)^d$ we define
 \begin{equation} \label{eq:psih}
    \Psi_h^\Gamma(x) := x + d_h(x) G_h(x) \quad \text{for}~x \in T \in \TGamma,
  \end{equation}
which approximates the function $\Psi^\Gamma$ defined in \eqref{psi1}.
To remove possible discontinuities of $\Psi_h^\Gamma$ in $\OGamma$ we apply the projection to obtain
\begin{equation}
  \thetahGamma:= P_h^\Gamma \Psi_h^\Gamma = \id + P_h^\Gamma(d_h G_h).
\end{equation}

Using the results in the Lemmas \ref{lem6} and \ref{propertiesdh} one can  derive estimates on the difference $\Psi^\Gamma- \Psi_h^\Gamma$  and $\thetahGamma- \Psi_h^\Gamma$. Such results, which are needed in the error analysis, are presented in the next two lemmas.

\begin{lemma} \label{lem3}
The estimate
\begin{equation} \label{estPsi}
  \max_{T \in \TGamma} \|\Psi^\Gamma - \Psi_h^\Gamma\|_{\infty,T} + h \max_{T \in \TGamma} \|D (\Psi^\Gamma - \Psi_h^\Gamma) \|_{\infty,T}  \lesssim h^{k+1}
\end{equation}
holds.
\end{lemma}
\begin{proof}
In the Appendix, section \ref{proof:lem3}, we give a proof which is based on the definitions of $\Psi_h^\Gamma$, $\Psi^\Gamma$ and the properties of $d_h$ and $d$ in Lemmas \ref{propertiesdh} and \ref{lem6}.
\end{proof}
\\[1ex]
Next, we compare $\Psi^\Gamma$ with the isoparametric mapping $\Theta_h^\Gamma= P_h \Psi_h^\Gamma$. 
\begin{lemma}\label{lem4}
The estimate
\begin{equation}
  \sum_{r=0}^{k+1} h^r \max_{T \in \TGamma}\Vert D^r( \thetahGamma - \Psi^\Gamma) \Vert_{\infty,T} 
\lesssim h^{k+1}
\end{equation}
holds.
\end{lemma}
\begin{proof} We have with $P_h v = I_k v $ for $v \in C(\OGamma)$ 
\begin{align*}
\|\thetahGamma - \Psi^\Gamma\|_{\infty, \OGamma} & \leq  
\|P_h (\Psi_h^\Gamma - \Psi^\Gamma) \|_{\infty, \OGamma}  + \| I_k \Psi^\Gamma - \Psi^\Gamma\|_{\infty, \OGamma} \\
& \lesssim \| \Psi_h^\Gamma - \Psi^\Gamma \|_{\infty, \OGamma}  + h^{k+1} \max_{T \in \TGamma}\| d G\|_{H^{k+1,\infty}(T)}. 
\end{align*}
With \eqref{resd4c} and the smoothness of $G$ we have that $\max_{T \in \TGamma} \| d G\|_{H^{k+1,\infty}(T)}$ is uniformly bounded. Further we have due to \eqref{estPsi} $\| \Psi_h^\Gamma - \Psi^\Gamma \|_{\infty, \OGamma} \lesssim h^{k+1}$, which proves the estimate for the $r=0$ term in the sum.  For $r=1, \ldots,k+1$ terms we note
\begin{align*}
h^r \Vert D^r(\thetahGamma - \Psi^\Gamma) \Vert_{\infty,T} & \leq  h^r \Vert D^r(\thetahGamma - I_k \Psi^\Gamma) \Vert_{\infty,T} + h^r \Vert D^r (I_k \Psi^\Gamma - \Psi^\Gamma) \Vert_{\infty,T} \\
 & \lesssim \Vert \thetahGamma - I_k \Psi^\Gamma \Vert_{\infty,T} + h^r \Vert D^r (I_k \Psi^\Gamma - \Psi^\Gamma) \Vert_{\infty,T} \\
 & \lesssim \Vert \thetahGamma - \Psi^\Gamma \Vert_{\infty,T} + \Vert \Psi^\Gamma - I_k \Psi^\Gamma \Vert_{\infty,T} + h^r \Vert D^r (I_k \Psi^\Gamma - \Psi^\Gamma) \Vert_{\infty,T}, 
\end{align*}
with estimates that are uniform in $T \in \TGamma$. 
The bound for the first term on the right hand side follows from the previous result. The other terms are also uniformly bounded by $h^{k+1}$ due to the regularity of $\Psi^\Gamma$, cf.~\eqref{fistderPsiA}.
\end{proof}
\ \\[1ex]
From the result in Lemma~\ref{lem4} we obtain an optimal convergence order result for the distance between the approximate interface $\Gamma_h = \thetahGamma(\Gammalin)$ and $\Gamma$.
 \begin{lemma}\label{estGradA}
  The estimate 
 \[
  {\rm dist}(\Gamma_h,\Gamma) \lesssim h^{k+1}
 \]
 holds.
 \end{lemma}
 \begin{proof}
 Take $x \in \Gamma$.
 Let $y=x+\xi \nabla \phi(x) \in \Gamma_h=\thetahGamma(\Gammalin)$ be the point closest to $x$.
 Using \eqref{eq1} and $\|\nabla \phi(x)\|_2 \sim 1$ we get, using $\phi(\PsiGamma(\Gammalin))=0$,
 \begin{align*}
  \|x-y\|_2 & \sim |\xi| \sim |\phi(x+\xi \nabla \phi(x)) - \phi(x)|= |\phi(x+\xi \nabla \phi(x))| \leq \|\phi\|_{\infty,\Gamma_h} \\
  & = \|\phi \circ\thetahGamma\|_{\infty, \Gammalin} = \|\phi\circ\thetahGamma- \phi\circ\PsiGamma\|_{\infty, \Gammalin} \lesssim \|\thetahGamma -\PsiGamma\|_{\infty, \OGamma} \lesssim h^{k+1}.
 \end{align*}
 \end{proof}

\subsection{Extension procedure} \label{sectextension}
In this section we explain and analyze a general extension procedure for extending piecewise smooth functions given on $\partial \OGamma$ to the domain $\Omega \setminus \OGamma$, in  such a way that the extension is zero on $\Omega \setminus \OGammaplus$ and piecewise smooth on $\OGammaplus \setminus \OGamma$. 
We use a \emph{local} extension procedure, introduced and analyzed in  \cite{lenoir1986optimal,bernardi1989optimal}, which is a standard tool in isoparametric finite element methods. In the sections~\ref{sec:globaltrafoPsi} and \ref{sec:globaltrafo} this procedure is applied to construct extensions of $\Psi^\Gamma$ and $\thetahGamma$.

We first describe an extension operator, introduced in \cite{lenoir1986optimal}, which handles the extension of a function from (only) one edge or face to a triangle or tetrahedron. The presentation here is simpler as in \cite{lenoir1986optimal}, because we restrict to the case $d=2,3$ (although this is not essential) and  we only treat extension of functions given on the piecewise linear boundary $\partial\OGamma$, hence we do not have to consider the issue of boundary parametrizations. 
Let $F$ be an edge or a face of $T\in \TGammaplus\setminus \TGamma$ from which we want to extend a function $w \in C(F)$, which is zero at the vertices of $F$,  to the interior of $T$. ${F}$ is called a $1$-face of $T$  if $F$ is an edge of the triangle $T$ ($d=2$) or an edge  the tetrahedron ${T}$ ($d=3$) and it is called a $2$-face if $F$ is a face of the tetrahedron ${T}$. Applying affine linear transformations $\hat \Phi_F : \hat{F} \rightarrow F$ and  $\hat \Phi_T : \hat{T} \rightarrow T$ we consider the extension problem in the reference configuration with $\hat{T}$  and $\hat{F}$ the reference element and a corresponding face (or edge) and $\hat{w} := w \circ \hat\Phi_F^{-1}$.
  By $\lambda_i,~i=1,..,d+1$, we denote the barycentric coordinates of $\hat{T}$ and  we assume that the vertices corresponding to the coordinates $\lambda_1, \ldots, \lambda_{p+1}$ with coordinates $a_1,...,a_{p+1} \in \rr^d$ are also the vertices of the $p$-face $\hat{F}$.
   The linear scalar weight function $\omega$ and the vector function $Z$,  mapping from $\hat{T}$ to $\hat{F}$, are defined by
\begin{equation} \label{defomega} 
  \omega := \sum_{i=1}^{p+1} \lambda_i \quad \text{and} \quad
  Z := \left( \sum_{i=1}^{p+1} \lambda_i a_i \right) / \omega.
\end{equation}
Furthermore, given the interpolation operator $\Lambda_l : C(\hat{F}) \rightarrow \mathcal{P}^l(\hat{F})$ we define
\begin{equation} \label{defA} 
  A_l^\ast := \id - \Lambda_l, \quad  A_l := \Lambda_l - \Lambda_{l-1} = - A_{l}^\ast + A_{l-1}^\ast.
\end{equation}
The interpolation operator $\Lambda_l$  in \cite{lenoir1986optimal} is the usual nodal interpolation operator. We note however that this choice is not crucial. Given these components we define the following 
extension operator from \cite{lenoir1986optimal,bernardi1989optimal}:
\begin{equation} \label{eq:lenoirextref}
  \mathcal{E}^{\hat{F}\rightarrow \hat{T}} \hat w :=\omega^{k+1} A_k^\ast(\hat w)\circ Z+ \sum_{l=2}^k \omega^l  A_l( \hat{w}) \circ Z , \quad \hat{w} \in C(\hat{F}).
\end{equation}
The extension in physical coordinates is then given by
\begin{equation}\label{eq:lenoirext}
  \mathcal{E}^{ F \rightarrow T} w := (\mathcal{E}^{\hat{F}\rightarrow \hat{T}} (w \circ \hat{\Phi}_F^{-1})) \circ \hat{\Phi}_T^{-1}.
\end{equation}
A few elementary properties of this extension operator are collected in the following lemma.
These results easily follow from Remarks 6.4-6.6 in \cite{bernardi1989optimal}. We use the notation $C_0^m(F):=\{\, v \in C^m(F)~|~v|_{\partial F}=0\}$.
We write $C_0(F):=C_0^0(F)$.
\begin{lemma} \label{propextension}
The following holds:
\begin{align}
  (\mathcal{E}^{ F \rightarrow T} w)_{|F} & = w_{|F} &\hspace*{-1.5cm}& \text{for all}~w \in C_0(F),\\
  \mathcal{E}^{ F \rightarrow T} w & \in \mathcal{P}^k(T)&\hspace*{-1.5cm}& \text{for all}~w\in \mathcal{P}^k(F), \label{proppol} \\
  \mathcal{E}^{ F \rightarrow T} w & = 0 &\hspace*{-1.5cm}& \text{for all}~w\in \mathcal{P}^1(F). \label{proppone} \\
  \intertext{
For all $p$-faces $F, \tilde F$ of $T$ with $\tilde F \neq F$:}
(\mathcal{E}^{ F \rightarrow T}w)_{|\tilde F} & = 0 &\hspace*{-1.5cm}& \text{for all $w \in C_0(F)$}.
\label{proponeface} \\
\intertext
{
  Let $T,\tilde T$ be two tetrahedra with a common face $F_2= T \cap \tilde T$ and $F_1$ an edge of $F_2$. Then:
  }
  (\mathcal{E}^{ F_1 \rightarrow T}w)_{|F_2} & =  (\mathcal{E}^{ F_1 \rightarrow \tilde T}w)_{|F_2} &\hspace*{-1.5cm}& \text{for all $w \in C_0(F_1)$}. \label{propconsist}
\end{align}
\end{lemma}
The following result is a key property of this extension operator. Similar results are derived in \cite{lenoir1986optimal,bernardi1989optimal}.
\begin{lemma} \label{lem:lenoirext}
  For any $p$-face $F$ of $T$  the following holds:
  $$
  \Vert D^n \mathcal{E}^{F\rightarrow T} w \Vert_{\infty,T} \lesssim 
 \sum_{r=n}^{k+1} h^{r-n} \Vert D^r w \Vert_{\infty,F}, \quad \forall ~w \in C_0^{k+1}(F),~ n=0,..,k+1.
  $$
\end{lemma}
\begin{proof}
  The proof is along the same lines as the proof of Theorem 6.2 in \cite{bernardi1989optimal}, and is given in the appendix, Section \ref{proof:lenoirext}.
\end{proof}\\[1ex]
We now describe how to obtain a ``global'' extension $w^{ext}$ of a function $w \in C(\partial\OGamma)$. The structure of this extension is as follows. First, we introduce a simple variant of the previous extension operator $\mathcal{E}^{F\rightarrow T}$ to extend values from vertices and obtain $w_0^{ext}$. The difference $w - w_0^{ext}$ has, by construction, zero values on all vertices, so that we can apply the extension from edges to elements, using $\mathcal{E}^{F\rightarrow T}$ on each element, to obtain $w_1^{ext}$. In two dimensions this already concludes the extension and we set $w^{ext} = w_1^{ext}$. In three dimensions we finally apply the extension from faces to elements, based on $w - w_1^{ext}$ which has zero values on all edges, and obtain the extended function $w^{ext}$.
\\[1ex] 
{\bf The extension from vertices}\\
We define the linear interpolation $I_1^{ext}:\, C(\partial \OGamma) \to V_h^1(\OGammaplus \setminus \OGamma)$ as
\begin{equation}
 \begin{cases}
  (I_1^{ext}w)(x_i) & = w(x_i) \quad \text{for all vertices}~ x_i \in \partial \OGamma \\
  (I_1^{ext}w)(x_i) & = 0 \quad \text{at all other vertices in}~\OGammaplus \setminus \OGamma,
 \end{cases}
\end{equation}
and for $w \in C(\partial \OGamma)$ we set $w_0^{ext} := I_1^{ext} w$. 
With this extension we have $(w - w_0^{ext})(x_i) = 0$ for all vertices $x_i \in \partial \OGamma$.
On all elements $T \in \TGammaplus \setminus \TGamma$ that have no edge in $\partial \OGamma$ we set $w^{ext} = w_0^{ext}$. All other elements are treated below.\\[1ex]
{\bf The extension from edges}\\
Let $S_1$ be the set of edges in $\partial \OGamma$ and $\T(F)$ the set of elements $T \in \TGammaplus \setminus \TGamma$ which have $F$ as an edge.
We define the extension from edges as
\begin{equation} \label{eq:extfromedge}
  w_1^{ext} := w_0^{ext} + \sum_{F \in S_1} \sum_{T\in\T(F)} \mathcal{E}^{F\rightarrow T} (w - w_0^{ext}).
\end{equation}
We note that $w_1^{ext}$ is continuous due to the consistency property \eqref{propconsist}.
On all elements $T \in \TGammaplus \setminus \TGamma$ that have no face in $\partial \OGamma$, i.e. all elements in the two dimensional case, we set $w^{ext} := w_1^{ext}$. All other elements are treated below with an additional extension from faces to elements.\\[1ex]
{\bf The extension from faces}\\
Let $S_2$ be the set of faces in $\partial \OGamma$ and $\T(F)$ the set of elements $T \in \TGammaplus \setminus \TGamma$ which have $F$ as a face.
Analogously to \eqref{eq:extfromedge} we define the extension from faces as
\begin{equation}
w_2^{ext} := w_1^{ext} + \sum_{F \in S_2} \sum_{T\in\T(F)} \mathcal{E}^{F\rightarrow T} (w - w_1^{ext})
\end{equation}
With this extension we finally have $(w - w_2^{ext})|_{\partial \OGamma} = 0$ and set $w^{ext} = w_2^{ext}$.\\[2ex]
\noindent
From the properties of the extension operator listed above it follows that $w^{ext} \in C(\OGammaplus\setminus \OGamma)$, $w^{ext}$ is a continuous extension of $w$ and $w=0$ on $\partial \OGammaplus$. This defines  the linear extension operator $\mathcal{E}^{\partial\OGamma}:\, C(\partial\OGamma) \to C(\OGammaplus\setminus \OGamma)$,  by $w^{ext}= \mathcal{E}^{\partial\OGamma} w$.
This operator is suitable for the extensions of $\thetahGamma - \id$ and $\PsiGamma- \id$ and inherits the boundedness property given in Lemma~\ref{lem:lenoirext} which leads to the following result.

\begin{theorem} \label{exthm} Let $\V(\partial \OGamma)$ denote the set of vertices in $\partial \OGamma$ and $\F(\partial \OGamma)$ the set of all edges ($d=2$) or faces ($d=3$) in $\partial \OGamma$. The following estimates hold
\begin{align}
  \|D^n\mathcal{E}^{\partial\OGamma} w\|_{\infty,\OGammaplus\setminus\OGamma}   & \lesssim \max_{F \in \F(\partial \OGamma)} \sum_{r=n}^{k+1} h^{r-n} \|D^rw\|_{\infty,F} \nonumber \\ & \quad +\, h^{-n}\max_{x_i \in \V(\partial \OGamma)}|w(x_i)| ,\quad n=0,1, \label{estext1} \\
  \max_{T \in \OGammaplus\setminus\OGamma} \|D^n\mathcal{E}^{\partial\OGamma} w\|_{\infty,T} & \lesssim \sum_{r=n}^{k+1} h^{r-n} \|D^rw\|_{\infty,F}, \quad n=2,\ldots, k+1,\label{estext2}
\end{align}
for all $w \in C(\partial \OGamma)$ such that $w \in C^{k+1}(F)$ for all $F \in \F(\partial \OGamma)$.
\end{theorem}
\begin{proof}
  We note that due to \eqref{proppone} we have that $\mathcal{E}^{\partial \OGamma} = I_1^{ext} + \mathcal{E}^{\partial \OGamma}( \id - I_1^{ext})$. Furthermore, $(w - I_1^{ext}w)(x_i)=0$ for all vertices $x_i \in \partial\OGamma$ and $\mathcal{E}^{\partial \OGamma}( \id - I_1^{ext})$ is a composition of the element-local extension operator $\mathcal{E}^{F \rightarrow T}$ only. Hence, the results in Lemma~\ref{lem:lenoirext} imply corresponding results for $\mathcal{E}^{\partial \OGamma}( \id - I_1^{ext})$. 
For $r \geq 2$ we have $D^rI_1^{ext}w=0$ and furthermore
\begin{equation} \label{pll}
\|D^r( w-I_1^{ext}w)\|_{\infty,F} \lesssim \|D^rw\|_{\infty,F}, \quad \text{for}~0 \leq r \leq k+1, ~~\text{and}~F\in \F(\partial \OGamma).
\end{equation}
From this estimate and the result in 
  Lemma~\ref{lem:lenoirext} the estimate \eqref{estext2} follows. Using
  \[
  \| I_1^{ext}w\|_{\infty,\OGammaplus\setminus\OGamma}+h \|DI_1^{ext}w\|_{\infty,\OGammaplus\setminus\OGamma} \lesssim \max_{x_i \in \V(\partial \OGamma)}|w(x_i)|
  \]
in combination with the  result in 
  Lemma~\ref{lem:lenoirext} and \eqref{pll} yields the estimate \eqref{estext1}.
\end{proof}


\subsection{The global mesh transformation $\Psi$}  \label{sec:globaltrafoPsi}
Given the local transformation $\PsiGamma$ and the extension operator $\mathcal{E}^{\partial\OGamma}$  defined in section  \ref{sectextension} we define the following global continuous mapping on $\Omega$:
\begin{equation} \label{eq:globext}
    \mathcal{E} \Psi^\Gamma = \left\{ \begin{array}{rcl}
                              \PsiGamma & \text{ on} & \OGamma, \\
                              \id +\mathcal{E}^{\partial\OGamma}(\PsiGamma-\id) & \text{ on} & \OGammaplus \setminus \OGamma, \\
                              \id & \text{ on} & \Omega \setminus \OGammaplus,
                              \end{array}\right.
  \end{equation}
  and $\Psi:=\mathcal{E} \Psi^\Gamma$. For this global mapping the following bounds on derivatives are easily derived from the results already obtained for $\Psi^\Gamma$ and for the extension operator $\mathcal{E}^{\partial\OGamma}$.
  \begin{theorem} \label{ThmPsi}
   For $\Psi=\mathcal{E} \Psi^\Gamma$ the following holds:
   \begin{align} 
    \|\Psi - \id\|_{\infty,\Omega}+ h\|D\Psi - I\|_{\infty,\Omega} & \lesssim h^2, \label{bpsi1} \\ 
    \max_{T\in\T} \Vert D^l \Psi \Vert_{\infty,T} & \lesssim 1, \quad 0 \leq l \leq k+1. \label{bpsi2}
   \end{align}
  \end{theorem}
  \begin{proof}
  On  $\Omega \setminus \OGammaplus$ these results are trivial since $\Psi - I=0 $ . On $\OGamma$ we have $D\Psi - I= D\PsiGamma- I$ and the results are a direct consequence of Corollary~\ref{lem:boundpsig}. On  $\OGammaplus \setminus \OGamma$ the result in \eqref{bpsi2}  for $l=0,1$ follows from \eqref{bpsi1} and for $l \geq 2$ it follows from \eqref{estext2} combined with \eqref{fistderPsiA}. We use the estimate in \eqref{estext1} and thus get
  \begin{align*}
    & \|\Psi - \id \|_{\infty,\OGammaplus\setminus\OGamma} + h\|D\Psi - I\|_{\infty,\OGammaplus\setminus\OGamma} \\
&  =\|\mathcal{E}^{\partial\OGamma}(\PsiGamma-\id)\|_{\infty,\OGammaplus\setminus\OGamma}+ h \|D\mathcal{E}^{\partial\OGamma}(\PsiGamma-\id)\|_{\infty,\OGammaplus\setminus\OGamma} \\
    & \lesssim \sum_{r=0}^{k+1} h^{r} \max_{F \in \F(\partial \OGamma)}\|D^r (\PsiGamma-\id)\|_{\infty,F}+  \max_{x_i \in \V(\partial \OGamma)}|(\PsiGamma-\id)(x_i)| \\
    & \lesssim \|\PsiGamma - \id\|_{\infty,\OGamma}+h  \|D\PsiGamma - I\|_{\infty,\OGamma} + h^2 \max_{2 \leq l \leq k+1} \max_{T \in \TGamma} \|D^l \PsiGamma\|_{\infty,T} + \max_{x_i \in \V(\partial \OGamma)}|d(x_i)|,
    \end{align*}
 and using \eqref{resd4a0} and the results of Corollary~\ref{lem:boundpsig}  completes the proof.
  \end{proof}\\[1ex]
  From \eqref{bpsi1} it follows that, for $h$ sufficiently small, $\Psi$ is a bijection on $\Omega$. Furthermore this mapping induces a family of (curved) finite elements that is regular of order $k$, in the sense as defined in \cite{bernardi1989optimal}. The corresponding  curved finite element space is given by 
  \begin{equation} \label{FEcurvedPsi}
    V_{h,\Psi}:= \{\, v_h \circ \Psi_h^{-1}~|~ v_h \in V_h^k\, \}.
  \end{equation}
Due to the results in Theorem~\ref{ThmPsi} the analysis of the approximation error for this finite element space as developed in \cite{bernardi1989optimal} can be applied. Corollary 4.1 from that paper yields that there exists an interpolation operator $\Pi_h:\, H^{k+1}(\Omega) \to  V_{h,\Psi}$ such that
\begin{equation} \label{errorBernardi}
 \|u- \Pi_h u\|_{L^2(\Omega)}+ h\|u-\Pi_h u\|_{H^1(\Omega)} \lesssim h^{k+1} \|u\|_{H^{k+1}(\Omega)} \quad \text{ for all } u \in H^{k+1}(\Omega).
\end{equation}
This interpolation result will be used in the error analysis of our method in section~\ref{approx}.
  \subsection{The global mesh transformation $\thetah$} \label{sec:globaltrafo}
  We define the extension of $\thetahGamma$ by using the same approach as for $\Psi^\Gamma$ in \eqref{eq:globext}, i.e. $\thetah:=\mathcal{E}\thetahGamma$. This global mapping $\thetah$ is used in the isoparametric unfitted finite element method explained in the next section. Note that $\Psi:=\mathcal{E}\Psi^\Gamma$ is only used in the analysis.
  \begin{remark}\rm The mapping 
    $\thetahGamma$ is a piecewise polynomial function of degree $k$ on $\partial\OGamma$ and thus $\thetah \in V_h^k(\Omega)^d$, cf. \eqref{proppol}. Further we have $\thetahGamma(x_i) = 0$ on all vertices of $\partial \OGamma$. Both lead to simplifications in the extension of $\thetahGamma$.
    In \eqref{eq:lenoirextref} the first term involving $A_k^{\ast}$ vanishes as $\thetahGamma|_F \in \mathcal{P}^k(F)$ for every edge and face $F \in \partial \OGamma$. Furthermore, $I_1^{ext} \thetahGamma = 0$ so that in an implementation only the extensions from edges and faces (with $w_0^{ext} = 0$) have to be considered. We do not address further implementation aspects of the mapping $\thetah$ here. For the local mapping $\thetahGamma$ these are discussed in \cite{lehrenfeld15}. The extension procedure is the same as  in the well-established isoparametric finite element method for a high order boundary approximation. Here we only note that in our implementation of this extension  we use a convenient approach based on a hierarchical/modal basis for the finite element space $V_h^k$, cf. \cite{karniadakis2013spectral}.
\end{remark}\\[1ex]
We discuss \emph{shape regularity} of the mapping $\Theta_h$.
Clearly, the mapping $\Theta_h$ should be a bijection on $\Omega$ and the transformed simplices $\Theta_h(T)$, $T \in \T$, should have some shape regularity property.

It is convenient to relate the transformed simplices $\Theta_h(T)$ to (piecewise) transformations of the unit simplex, denoted by $\hat T$.
The simplicial triangulation $\T$ can be represented by affine transformations $\Phi_T(x)=A_Tx +b_T$, i.e.
$\T=\{\, \Phi_T(\hat T)\,\}$.
This defines a mapping $\hat T \to \tilde{T}:=\Theta_h (\Phi_T(\hat T))$.
This mapping should be bijective and well conditioned.
Note that $\kappa(D(\thetah \circ \Phi_T)) \leq \kappa(D \thetah) \kappa(D \Phi_T)$, where $\kappa(\cdot)$ denotes the spectral condition number.
Since $\Phi_T: \hat T \to T$ is bijective and  well-conditioned 
it suffices to show the bijectivity and well-conditioning of $\Theta_h: T \to \Theta_h(T)$. Note that $\Psi- \thetah= \mathcal{E}(\PsiGamma-\thetahGamma)$ and using the estimate   \eqref{estext1}, with $n=1$, combined with the result in Lemma~\ref{lem4}, we get $\|D(\Psi- \thetah)\|_{\infty,\Omega} \lesssim h$. Thus, combined with the estimate in    \eqref{bpsi1} we get
\begin{equation} \label{estD}
  \|D\thetah - I\|_{\infty,\Omega} \lesssim h.
\end{equation}
This implies that for $h$ sufficiently small (the ``resolved'' case) $D\Theta_h$ is invertible and thus $\Theta_h: T \to \Theta_h(T)= \tilde T$ is a bijection and furthermore $ \kappa(D \Theta_h) = 1+ \mathcal{O}(h)$. Hence, we have shape regularity of $\tilde{\T} = \{\tilde T\}$ for $h$ sufficiently small. In the analysis we consider only the resolved case with $h$ sufficiently small. In practice one  needs suitable   modifications of the method  to guarantee shape regularity also in cases where $h$ is not ``sufficiently small''. We do not discuss this here and instead refer to \cite{lehrenfeld15}.
Further important consequences of the smallness of the deformation $\thetah$ are summarized in the following lemma.
\begin{lemma}\label{lemF}
Let $F := D \thetah$, $J_V := \det(F)$, $ J_{\Gamma} := J_V \Vert F^{-T} n_{\Gamma^{\rm lin}} \Vert_2$. There holds:
\begin{subequations}
\begin{align}
 \max_{T \in \T} \| F - I \|_{\infty,T} & \lesssim h, \label{estF}\\
         \int_{\Omega_i^{\rm lin}} v^2 \, d\tilde x \sim \int_{\Omega_i^{\rm lin}} J_V v^2 \, d\tilde x &= \int_{\thetah(\Omega_i^{\rm lin})} \!\!\!\!\!\!\!\! (v \circ \thetah^{-1})^2 \, d x,   \label{jacV}\\
      \label{jacG}
      \int_{\Gamma^{\rm lin}} v^2 \, d\tilde s \sim \int_{\Gamma^{\rm lin}} J_\Gamma v^2 \, d\tilde s &= \int_{\thetah(\Gamma^{\rm lin})} \!\!\!\!\!\!\!\! (v \circ \thetah^{-1})^2 \, ds, \\
      \label{jacgrad}
   \Vert \nabla (v \circ \thetah^{-1}) \Vert_2  & \sim \Vert \nabla v \Vert_2 \text{ a.e. in } \Omega \text{ for } v \in H^1(\Omega).
    \end{align}
\end{subequations}
\end{lemma}
\begin{proof} 
From  \eqref{estD}  the result \eqref{estF} and $\det(F)=1+\mathcal{O}(h)$ easily follow. The latter implies the $\sim$ result in  \eqref{jacV}. We now consider the integral transformation result in \eqref{jacG}. We only treat $d=3$ ($d=2$ is very similar). Take $x \in \Gammalin \cap T$ and a local orthonormal system $t_1, t_2 \in n_{\Gammalin}(x)^\perp$ at $x$. The function $g(z_1,z_2):= \thetah(x+z_1t_1+z_2 t_2)$, $z_i \in \Bbb{R}$ parametrizes $\thetah(\Gammalin \cap T)$. For the change in measure we have, with $F=F(x)=D\thetah(x)$, 
\begin{equation} \label{changemeas} \begin{split}
 ds & = \big\| \frac{\partial g}{\partial z_1} \times  \frac{\partial g}{\partial z_2}\big\|_2\, d \tilde s =\|F t_1 \times F t_2\|_2\, d\tilde s \\ & = |\det (F)| \|F^{-T}(t_1 \times t_2)\|_2\, d\tilde s =
\det (F) \|F^{-T}n_{\Gammalin}(x)\|_2 \, d\tilde s = J_\Gamma \, d\tilde s ,  
\end{split} \end{equation}
which proves the equality in \eqref{jacG}.
From  \eqref{estF}  and $\det(F)=1+\mathcal{O}(h)$ we get the $\sim$ result in  \eqref{jacG}. With $\nabla (v \circ \thetah^{-1})= F^{-T} \nabla v$ and \eqref{estF} the result in  \eqref{jacgrad} follows.
\end{proof}

\section{Isoparametric unfitted finite element method} \label{unfittedFEM}
In this section we introduce the isoparametric unfitted finite element method based on the isoparametric mapping $\Theta_h$.
We consider the model elliptic interface problem \eqref{eq:ellmodel}.
 The weak formulation of this problem is as follows: determine $u \in  H_0^1(\Omega)$ such that
\[
 \int_{\Omega} \alpha \nabla u \cdot \nabla v \, dx = \int_\Omega f v \, dx \quad \text{for all}~~v \in H_0^1(\Omega).
\]
We define the isoparametric Nitsche unfitted FEM  as a transformed version of the original Nitsche unfitted FE discretization \cite{hansbo2002unfitted} with respect to the interface approximation $\Gamma_h = \Theta_h(\Gammalin)$.  We introduce some further notation.
The standard unfitted space w.r.t. $\Gammalin$ is denoted by
\[
  V_h^\Gamma:= {V_h^k}_{|\Omegalin_1}\oplus {V_h^k}_{|\Omegalin_2}.  
\]
To simplify the notation we do not explicitly express the polynomial degree $k$ in $V_h^\Gamma$. 
\begin{remark} \rm Note that the polynomial degree $k$ is used in  the finite element space that contains the level set function approximation, $\phi_h \in V_h^k$.  In the definition of the unfitted finite element space $V_h^\Gamma$ above, which is used for the discretization of the interface problem \eqref{eq:ellmodel},   we could also use a polynomial degree $m \neq k$. We restrict to the case that both spaces (for the discrete level set function and for the discretization of the PDE) use the same degree $k$, because this simplifies the presentation and there is no significant improvement of the method if one allows $m \neq k$.
\end{remark}
\ \\[1ex]
The isoparametric unfitted FE space is defined as
\begin{equation}\label{transfspace}
 V_{h,\Theta}^\Gamma:= \{\, v_h \circ \Theta_h^{-1}~|~ v_h \in V_h^\Gamma\, \}= \{\, \tilde v_h~|~\tilde v_h \circ \Theta_h \in  V_h^\Gamma\, \}.
\end{equation}
Based on this space we formulate a discretization of \eqref{eq:ellmodel} using the Nitsche technique \cite{hansbo2002unfitted} with $\Gamma_h = \Theta_h(\Gammalin)$ and $\Omega_{i,h} = \Theta_h(\Omegalin_i)$ as numerical approximation of the geometries: determine $ u_h \in V_{h,\Theta}^\Gamma$ such that
\begin{equation} \label{Nitsche1}
 A_h(u_h,v_h) := a_h(u_h,v_h) + N_h(u_h,v_h) = f_h(v_h) \quad \text{for all } v_h \in V_{h,\Theta}^\Gamma
\end{equation}
with the bilinear forms
\begin{subequations} \label{eq:blfs}
\begin{align}
a_h(u,v) & := \sum_{i=1}^2 \alpha_i \int_{\Omega_{i,h}} \nabla u \cdot \nabla v dx, \\
N_h(u,v) & := N_h^c(u,v) + N_h^c(v,u) + N_h^s(u,v),\\
N_h^c(u,v) & := \int_{\Gamma_h} \average{-\alpha \nabla v} \cdot n \spacejump{u} ds, \quad 
N_h^s(u,v) := \bar \alpha \frac{\lambda}{h} \int_{\Gamma_h} \spacejump{u} \spacejump{v} ds
\end{align}
\end{subequations}
for $u, v \in V_{h,\Theta}^\Gamma + \Vreg$ with $\Vreg := H^1(\Omega) \cap H^2(\Omega_{1,h} \cup \Omega_{1,h})$.

Here, $ n = n_{\Gamma_h}$ denotes the outer normal of $\Omega_{1,h}$ and 
$\bar \alpha = \frac12( \alpha_1 + \alpha_2)$ the mean diffusion coefficient. 
For the averaging operator  $\average{\cdot}$ there are different possibilities. We use 
 $\averageleft w \averageright := \kappa_1 w_{|\Omega_{1,h}} + \kappa_2w_{|\Omega_{2,h}} $ with a ``Heaviside'' choice where $\kappa_i = 1$ if $|T_i| > \frac12 |T|$ and $\kappa_i = 0$ if $|T_i| \leq \frac12 |T|$. Here, $T_i = T \cap \Omegalin_i$, i.e. the cut configuration on the undeformed mesh is used. This choice in the averaging renders the scheme in \eqref{Nitsche1} stable (for sufficiently large $\lambda$) for arbitrary polynomial degrees $k$, independent of the cut position of $\Gamma$, cf. Lemma~\ref{lemcoercive} below. A different choice for the averaging which also results in a stable scheme is $\kappa_i = |T_i|/|T|$.

In order to define  the  right hand side functional $f_h$ we first assume that the source term $f_i :\, \Omega_i \to \Bbb{R}$ in \eqref{eq:ellmodel1} is (smoothly) extended to $\Omega_{i,h}$, such that $f_i= f_{i,h}$ on $\Omega_i$ holds.  This extension is denoted by $f_{i,h}$. We define
\begin{equation} \label{eq:lfs}
  f_h(v) := \sum_{i=1,2} \int_{\Omega_{i,h}} f_{i,h} v dx.
\end{equation}
We define $f_h$ on $\Omega$ by ${f_h}_{|\Omega_{i,h}}:= f_{i,h}$, $i=1,2$.

For the implementation of this method, in the integrals we apply a transformation of variables $y:=\Theta_h^{-1}(x)$.
This results in the following representations of the bi- and linear forms:
\begin{subequations}  \label{atrans}
\begin{align}
 a_h(u,v)  & = \sum_{i=1,2} \alpha_i \int_{\Omegalin_i}  D\Theta_h^{-T} \nabla u \cdot  D\Theta_h^{-T} \nabla v ~ \det (D\Theta_h)\, dy, \\
 f_h(v)  & = \sum_{i=1,2} \int_{\Omegalin_i} (f_{i,h} \circ \Theta_h ) v ~ \det (D\Theta_h)\, dy, \\
N_h^c(u,v) & = \int_{\Gammalin} \det(D \Theta_h) D\Theta_h^{-T} \average{-\alpha \nabla v} \cdot (D\Theta_h^{-T} \cdot n^{\text{lin}}) \spacejump{u} dy, \label{q} \\ 
N_h^s(u,v) & = \bar \alpha \frac{\lambda}{h} \int_{\Gammalin} \mathcal{J}_\Gamma \spacejump{u} \spacejump{v} dy,
\end{align}
\end{subequations}
where $\mathcal{J}_{\Gamma}  = \det (D\Theta_h) \Vert D \Theta_h^{-T} \cdot n^{\text{lin}} \Vert$ is the ratio between the measures on $\Gamma_h$ and $\Gammalin$ and $n^{\text{lin}} = \nabla \hphi / \Vert \nabla \hphi \Vert$ is the normal to $\Gammalin$. We note that in \eqref{q} we exploited that the normalization factor $\Vert D \Theta_h^{-1} \cdot n^{\text{lin}} \Vert$ for the normal direction $n$ of $\Gamma_h$ cancels out with the corresponding term in $\mathcal{J}_{\Gamma}$.
Based on  this transformation the implementation of integrals is carried out as for the case of the piecewise planar interface $\Gammalin$. The additional variable coefficients $D \thetah^{-T}$, $\det(D \thetah)$  are easily and efficiently computable using the property that $\thetah$ is a finite element (vector) function. 
The integrands in \eqref{atrans} are in general not polynomial so that exact integration can typically not be guaranteed. For a discussion of the thereby introduced additional consistency error we refer to Remark \ref{rem:quad} in the analysis.
 
\subsection{Numerical experiment} \label{sec:numex}

In this section we present results of a numerical experiment for the method \eqref{Nitsche1}.
The domain is $\Omega = [-1.5,1.5]^2$, and the interface is given as $\Gamma = \{ \phi(x) = 0 \}$ with the level set function $\phi(x) = \Vert x \Vert_4 - 1$. Here  $\Vert x \Vert_4 := (\sum_{i=1}^d x_i^4)^{\frac14}$. The zero level of $\phi$ describes a ``smoothed square'' and $\phi$ is equivalent to a signed distance function in the sense that $ \sqrt[4]{1/2} \leq \Vert \nabla \phi \Vert_2 \leq 1$ such that $\mathrm{dist}(x,\Gamma) \leq \sqrt[4]{2} \vert \phi(x) \vert\text{ for all }x\in\Omega$.
The level set function $\phi$ is approximated with $\phi_h \in V_h^k$ by interpolation. 
For the problem in \eqref{eq:ellmodel}, we take the diffusion coefficient $(\alpha_1,\alpha_2) = (1,2)$ and Dirichlet boundary conditions and right-hand side $f$ such that the solution is given by 
\begin{equation} \label{eq:solution}
 u(x) = \left\{ \begin{array}{rc} 1 + \frac{\pi}{2} - \sqrt{2} \cdot \cos(\frac{\pi}{4} \Vert x \Vert_4^4), & \ x \in \Omega_1, \\
\frac{\pi}{2} \Vert x \Vert_4\hphantom{)}, & \ x \in \Omega_2.  \end{array} \right.
\end{equation}
We note that $u$ is continuous on $\Omega$ but has a kink across the interface $\Gamma$ such that $\alpha_1 \nabla u_1 \cdot n_\Gamma = \alpha_2 \nabla u_2 \cdot n_\Gamma$, cf. the sketch of the solution in Figure \ref{fig:sketchsol}. 

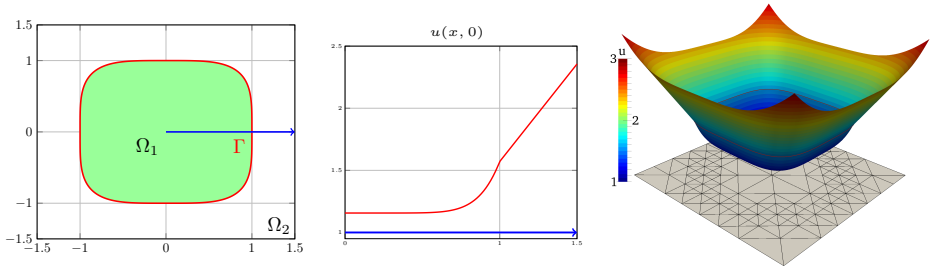
\begin{figure}
  \begin{center}
    \begin{minipage}{0.32\textwidth}
      \begin{center}
        \begin{tikzpicture}[scale=0.5]
          \begin{axis}[
            xmin=-1.5,xmax=1.5,
            ymin=-1.5,ymax=1.5,
            grid=both,
            xtick={-1.5,-1,0,1,1.5},
            ytick={-1.5,-1,0,1,1.5},
            ]
            \addplot [domain=0:360,samples=200,color=red,very thick,fill=green!40,opacity=1.0,draw opacity=1.0]
            (
            { abs(cos(x))/(cos(x))*sqrt(abs(cos(x))) },
            { abs(sin(x))/(sin(x))*sqrt(abs(sin(x))) }
            ); 

            \node [right,scale=1.5] at (axis cs:  0.7,  -0.2) {\color{red} $\Gamma$};
            \node [below left,scale=1.5] at (axis cs:  0.0,  0.0) { $\Omega_1$};
            \node [below right,scale=1.5] at (axis cs:  1.1,  -1.1) { $\Omega_2$};
            \coordinate (a) at (axis cs:  0.0,  0.0);
            \coordinate (b) at (axis cs:  1.5,  0.0);
            \draw[blue, very thick, ->] (a) -- (b);
          \end{axis}
        \end{tikzpicture}
      \end{center}
    \end{minipage}
    \begin{minipage}{0.28\textwidth}
      \begin{center}
        \tiny $u(x,0)$ \\ 
        \begin{tikzpicture}[scale=0.45]
          \begin{axis}[
            xmin=0,xmax=1.5,
            ymin=0.95,ymax=2.5,
            grid=both,
            axis line style={->},
            xtick={-1.5,-1,0,1,1.5},
           ytick={1,1.5,2,2.5},
            ]
            \addplot [domain=-1.5:1.5,samples=200,color=red,very thick,draw opacity=1.0]
            (
            { x },
            {  (abs(x)<1) * (1+pi/2.0-sqrt(2.0)*cos(45*x*x*x*x)) + (abs(x)>=1) * pi/2 * (abs(x)) }
            ); 
            \coordinate (a) at (axis cs:  0.0,  1.0);
            \coordinate (b) at (axis cs:  1.5,  1.0);
            \draw[blue, ultra thick, ->] (a) -- (b);
          \end{axis}
        \end{tikzpicture}
      \end{center}
    \end{minipage}
    \begin{minipage}{0.34\textwidth}
      \begin{center}
        \includegraphics[width=0.98\textwidth]{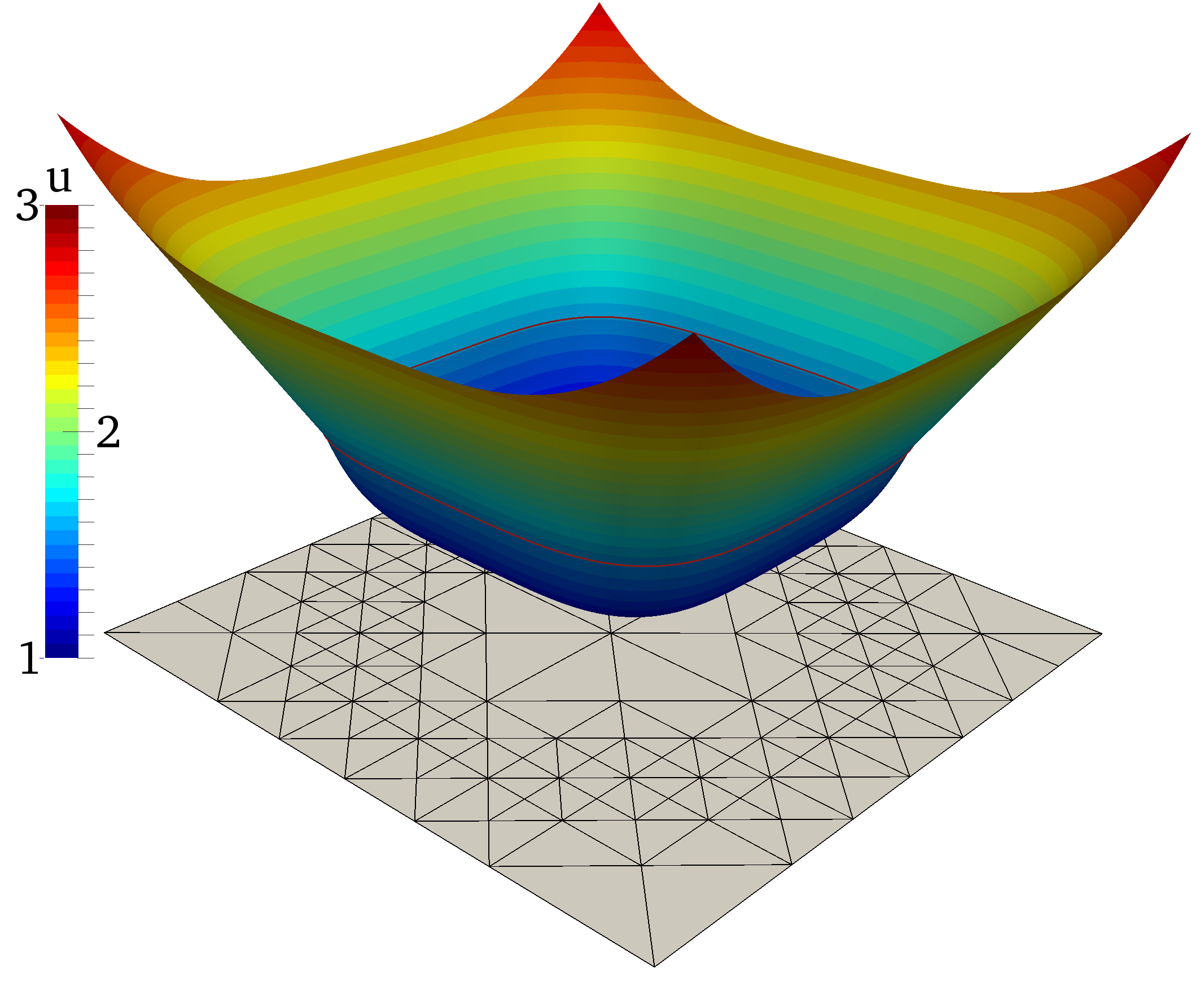}
      \end{center}
    \end{minipage}
  \end{center}
\vspace*{-0.2cm}
\caption{Sketch of domains $\Omega_1$, $\Omega_2$ (left), the solution (middle and right) and the initial mesh ($L=0$, right) for the example in Section \ref{sec:numex}.}
\label{fig:sketchsol}
\end{figure}

For the construction of the mapping $\thetah$, the one-dimensional nonlinear problem which results from \eqref{eq:psihmap} is solved with a Newton-iteration, cf. \cite[section 2.3]{lehrenfeld15}, up to a (relative) tolerance of $10^{-14}$, which required only $2-3$ iterations for each point. For the search direction  we chose $G_h = \nabla \phi_h$, cf.~\eqref{eq:gh}. We also considered $G_h=P_h^\Gamma \nabla \phi_h$, which led to very similar results.

For the computation of the integrals we apply the transformation to the reference domains $\Omegalin_i$ and $\Gammalin$ as in \eqref{atrans}. On each element of the reference domain (e.g. $\Omegalin_i \cap T$, $T \in \T_h$) we apply quadrature rules of exactness degree $2k-2$ for the integrals in $a_h(\cdot,\cdot)$ and $N_h^c(\cdot,\cdot)$ and quadrature rules of exactness degree $2k$ for the integrals in $N_h^s(\cdot,\cdot)$, see also remark \ref{rem:quad} below.

The method has been implemented in the add-on library \texttt{ngsxfem} to the finite element library \texttt{NGSolve} \cite{schoeberl2014cpp11}.

We chose an initial simplicial mesh ($L=0$) with $282$ elements which sufficiently resolves the interface such that shape regularity of the mesh (after transformation) is given without a further limitation step as discussed in \cite[section 2.6]{lehrenfeld15}. 
Starting from this initial triangulation uniform refinements are applied.
The stabilization parameter $\lambda$ in the unfitted Nitsche method is chosen as $\lambda = 20 \cdot k^2$.

With the discrete solution to \eqref{Nitsche1} $u_h \in V_{h,\Theta}^\Gamma$, and $u$ as in \eqref{eq:solution} we define the following error quantities which we can evaluate for the numerical solutions:
\begin{align*}
d_{\Gamma_h}\! :=\!  \Vert \phi \Vert_{\infty,\Gamma_h}, \quad e_{\jump{\cdot}}:= \Vert \jump{E_i u -  u_h} \Vert_{L^2(\Gamma_h)} \\
e_{L^2}^2\! :=\! \sum_{i=1,2} \Vert E_i u -  u_h\Vert_{L^2(\Omega_{i,h})}^2, \quad
e_{H^1}^2\! :=\! \sum_{i=1,2} \Vert \nabla(E_i u -  u_h) \Vert_{L^2(\Omega_{i,h})}^2.
\end{align*}
Here, $E_i:\Omega_i \rightarrow \Omega_{i,h}=\thetah(\Omegalin_i)$ denotes the canonical extension operator for the solution (using the representation in \eqref{eq:solution}). 
Due to the equivalence of $\phi$ to a signed distance function we have $\mathrm{dist}(\Gamma_h,\Gamma) \leq \sqrt[4]{2}~d_{\Gamma_h}$ and thus $d_{\Gamma_h}$ yields a (sharp) bound for the error in the geometry approximation.  The error quantities $e_{H^1}$, $d_{\Gamma_h}$ and $e_{\jump{\cdot}}$ are also used in the error analysis, cf.~\eqref{defn}.  The results in the experiment confirm the theoretical error bounds. We also include the $L^2$-error   $e_{L^2}$, although in the analysis in this paper we do not derive $L^2$-error bounds, cf.~Section~\ref{sectoutlook}.

\begin{table}[h!]
\small
\begin{center}
\begin{tabular}{r@{\ \ }r@{\ \ \ }
  r@{\ (}c@{)\ \ \ }
  r@{\ (}c@{)\ \ \ }
  r@{\ (}c@{)\ \ \ }
  r@{\ (}c@{)\ \ \ }
}
$k$ & $L$ & $d_{\Gamma_h}\qquad$  & eoc & $e_{L^2}\qquad$ & eoc & $e_{H^1}\qquad$ & eoc & $e_{\jump{\cdot}}\qquad$ & eoc\\
\toprule
1 & 0 & \num{  0.0050221} &  -  & \num{   0.0850176} &   - & \num{   0.687233} &   - & \num{ 0.00373823} &   - \\
  & 1 & \num{ 0.00223861} & 1.2 & \num{   0.0245449} & 1.8 & \num{    0.40476} & 0.8 & \num{ 0.00165237} & 1.2 \\
  & 2 & \num{ 0.00053897} & 2.0 & \num{  0.00611226} & 2.0 & \num{   0.209966} & 0.9 & \num{ 0.00038622} & 2.1 \\
  & 3 & \num{ 0.00012375} & 2.1 & \num{  0.00150264} & 2.0 & \num{   0.105699} & 1.0 & \num{0.000113794} & 1.8 \\
  & 4 & \num{3.36362e-05} & 1.9 & \num{ 0.000373006} & 2.0 & \num{  0.0534224} & 1.0 & \num{2.74341e-05} & 2.1 \\
  & 5 & \num{8.86244e-06} & 1.9 & \num{  9.3496e-05} & 2.0 & \num{  0.0267841} & 1.0 & \num{6.95582e-06} & 2.0 \\
  & 6 & \num{2.25885e-06} & 2.0 & \num{ 2.36659e-05} & 2.0 & \num{  0.0134293} & 1.0 & \num{1.75559e-06} & 2.0 \\
\midrule
2 & 0 & \num{6.993e-4} &  -  & \num{3.021E-3} &   - & \num{1.522E-1} &   - & \num{0.000113392} &   - \\
  & 1 & \num{6.514e-5} & 3.4 & \num{5.893E-4} & 2.4 & \num{5.152E-2} & 1.6 & \num{3.10019e-05} & 1.9 \\
  & 2 & \num{1.197e-5} & 2.4 & \num{7.899E-5} & 2.9 & \num{1.369E-2} & 1.9 & \num{3.65325e-06} & 3.1 \\
  & 3 & \num{2.111e-6} & 2.5 & \num{9.976E-6} & 3.0 & \num{3.446E-3} & 2.0 & \num{4.82193e-07} & 2.9 \\
  & 4 & \num{2.447e-7} & 3.1 & \num{1.265E-6} & 3.0 & \num{8.704E-4} & 2.0 & \num{5.48117e-08} & 3.1 \\
\midrule
3 & 0 & \num{  8.6985e-05} &   - & \num{0.001665610} &   - & \num{  0.0275621} &   - & \num{5.28899e-06} &   - \\
  & 1 & \num{ 1.96038e-05} & 2.1 & \num{0.000196828} & 3.1 & \num{ 0.00369957} & 2.0 & \num{5.32893e-07} & 3.3 \\
  & 2 & \num{ 1.25282e-06} & 4.0 & \num{1.17898e-05} & 4.1 & \num{0.000442553} & 3.1 & \num{2.00381e-08} & 4.7 \\
  & 3 & \num{ 6.39185e-08} & 4.3 & \num{7.14402e-07} & 4.0 & \num{5.27611e-05} & 3.1 & \num{ 1.4139e-09} & 3.8 \\
\midrule
4 & 0 & \num{ 2.01259e-05} &   - & \num{9.66239e-05} &   - & \num{ 0.00782523} &   - & \num{5.74363e-07} &   - \\
  & 1 & \num{ 1.02828e-06} & 4.3 & \num{2.55431e-06} & 5.2 & \num{0.000621062} & 3.7 & \num{3.22889e-08} & 4.2 \\
  & 2 & \num{ 2.51851e-08} & 5.4 & \num{9.79816e-08} & 4.7 & \num{ 3.5234e-05} & 4.1 & \num{1.06596e-09} & 4.9 \\
\midrule
5 & 0 & \num{ 2.47328e-06} &   - & \num{0.000106408} &   - & \num{ 0.00103943} &   - & \num{3.01411e-08} &   - \\
  & 1 & \num{  2.0333e-07} & 3.6 & \num{1.58967e-06} & 6.1 & \num{2.88465e-05} & 5.2 & \num{1.80272e-09} & 4.1 \\
\midrule
6 & 0 & \num{ 9.32775e-07} &   - & \num{1.49125e-06} &   - & \num{0.000254924} &   - & \num{ 7.1063e-09} &  -  \\
\bottomrule
\end{tabular}
\end{center}
\caption{Discretization errors and estimated orders of convergence (eoc) for the example in Section \ref{sec:numex}.}
\label{tab:numex2d}
\end{table}

Results are listed in Table \ref{tab:numex2d}. We observe  $d_{\Gamma_h} \sim \mathcal{O}(h^{k+1})$, $e_{L^2} \sim \mathcal{O}(h^{k+1})$, $e_{H^1} \sim \mathcal{O}(h^{k})$, i.e., optimal order of convergence in these  quantities. The observed convergence rate for the jump across the interface, $e_{\jump{\cdot}} \sim \mathcal{O}(h^{k+1})$, is better than predicted in the analysis below (by half an order). 
For a fixed mesh ($L=const$) we observe that increasing the polynomial degree $k$ dramatically decreases the error in all four error quantities. 
The discretization with $k=6$ on the coarsest mesh $L=0$ with only $6835$ unknowns (last row in  the table) is much more accurate in all four quantities than the discretization with $k=1$ after $6$ additional mesh refinements and $583401$ unknowns (row $k=1$, $L=6$ in the table). 
\begin{remark} \label{linsys}
\rm 
We note that the arising linear systems depend on the position of the interface within the computational mesh and can become arbitrarily ill-conditioned. To circumvent the impact of this we used a sparse direct solver for the solution of linear systems .
In our experiment the very poor  conditioning of the  stiffness matrix for high order discretizations limited the achievable accuracy to $e_{L^2} \approx \num{1e-6}$ and $e_{H^1} \approx \num{1e-6}$. Hence, we stopped the refinement if discretization errors in this order of magnitude have been reached. The solution of linear systems is an important issue for high order unfitted discretizations and requires further attention. To deal with the ill-conditioning one may consider (new) preconditioning strategies or further stabilization mechanisms in the weak formulation. We leave this as a topic for future research.
\end{remark}

\section{Error Analysis}\label{sec:erroranalysis}
The main new contribution of this paper is an error analysis of the method presented in Section~\ref{unfittedFEM}. 
We start with the definition of norms and derive properties of the discrete variational formulation in Section \ref{sec:normsplus}.
These results are obtained using the fact that the isoparametric method can be seen as a small perturbation of the standard unfitted Nitsche-XFEM method. To bound the error terms we use the bijective mapping on $\Omega$ given by $\Phi_h:= \Psi \circ \thetah^{-1}$. This mapping has the property $\Phi_h(\Gamma_h)=\Gamma$ and is close to the identity. Some further relevant properties of this mapping are treated in Section~\ref{sectmappsi}.
Using this mapping we  derive a Strang lemma type result which relates the discretization error to consistency errors in the bilinear form and right-hand side functional and to approximation properties of the finite element space, cf. Section~\ref{sectStrang}. 
In the Sections~\ref{sec:integrals} and ~\ref{approx} we prove bounds for the consistency and approximation error, respectively.  Finally, in Section~\ref{sectH1} an  optimal-order $H^1$-error bound is given.

\subsection{Norms and properties of the bilinear forms} \label{sec:normsplus}\ \\
In the error analysis we use the norm
\begin{align}
 \Vert v \Vert_{h}^2 & := \vert v \vert_{1}^2 + \Vert \jump{v} \Vert_{\frac12,h,\Gamma_h}^2 + \Vert \average{\alpha \nabla v} \Vert_{-\frac12,h,\Gamma_h}^2, \label{defn} \\
\text{ with } \Vert v \Vert_{\pm \frac12,h,\Gamma_h}^2 & := \left(\bar{\alpha} / h \right)^{\pm 1} \Vert v \Vert_{L^2(\Gamma_h)}^2 \text{ and } |v|_1^2:= \sum_{i=1,2} \alpha_i \Vert \nabla v \Vert_{L^2(\Omega_{i,h})}^2. \label{ppp}
\end{align}
Note that the norms are formulated with respect to $\Omega_{i,h} = \thetah(\Omegalin_i)$ and  $\Gamma_{h} = \thetah(\Gammalin)$ and include a scaling depending on $\alpha$.
\begin{remark} \rm
For simplicity we restrict to the setting of a quasi-uniform family of triangulations. In the more general case of a shape regular (not necessarily quasi-uniform) family of triangulations one has to replace the  interface norm in \eqref{ppp} by one with an element wise scaling:
\begin{equation} \label{defr} \Vert v \Vert_{\pm \frac12,h,\Gamma_h}^2  := \sum_{T \in \T^\Gamma} \left(\bar{\alpha} / h_T \right)^{\pm 1} \Vert v \Vert_{L^2(\Gamma_{h,T})}^2 \text{ with } \Gamma_{h,T} := \thetah(T\cap \Gammalin)
\end{equation}
\end{remark}
 
We further define the space of sufficiently smooth functions which allow for the evaluation of normal gradients at the interface $\Gamma_h$: 
\begin{equation} \label{defVreg}
\Vreg := H^1(\Omega) \cap H^2(\Omega_{1,h} \cup \Omega_{2,h}).
\end{equation}
Note that the norm $\|\cdot\|_h$ and the bilinear forms in \eqref{eq:blfs} are well-defined on $\Vreg+V_{h,\Theta}^\Gamma$. 
The stability of the method relies on an appropriate choice for the weighting operator $\average{u} = \kappa_1 u_1 + \kappa_2 u_2$ in \eqref{Nitsche1}. For the stability analysis of the weighting operator we introduce the criterion
\begin{equation} \label{eq:crit}
\kappa_i^2 \leq c_{\kappa} |T_i|/|T|,~i=1,2,
\end{equation}
with a constant $c_{\kappa}$ independent of $\alpha$, $h$ and the cut configuration. We note that this criterion is fulfilled for the Heaviside choice, cf. Section~\ref{unfittedFEM}, with $c_{\kappa}=2$ and the weighting $\kappa_i =|T_i|/|T|$ proposed in \cite{hansbo2002unfitted} with $c_{\kappa}=1$.
Using this criterion we derive the following inverse estimate, which is an important ingredient in the stability analysis of the unfitted Nitsche method.
\begin{lemma}\label{leminv}
On a shape regular (not necessarily quasi-uniform) family of triangulations 
with an averaging operator $\average{u}$ satisfying \eqref{eq:crit} there holds
\begin{equation} \label{inverse}
 \Vert \average{\alpha \nabla u} \Vert_{-\frac12,h,\Gamma_h} 
\lesssim  \vert u \vert_{1} \quad \text{for all}~ u \in V_h^\Gamma.
\end{equation}
\end{lemma}
\begin{proof}
It suffices to derive the localized estimate for an arbitrary element $T \in \T^\Gamma$. 
Note that for the general case of a not necessarily quasi-uniform family of triangulations the norm used on the left-hand side in \eqref{inverse} is as in \eqref{defr}. 
We only have to show
$$
h_T \kappa_i^2 \Vert p \Vert_{L^2(\Gamma_{h,T})}^2  \lesssim \Vert p \Vert_{L^2(\thetah(T_i))}^2 \quad \text{for all}~ p \circ \thetah \in \mathcal{P}^k(T)
$$
with $\Gamma_{h,T} := \thetah(\Gammalin_T)$, $\Gammalin_T := \Gammalin \cap T$ and $T_i := \Omegalin_i \cap T$ for $i=1,2$. With Lemma \ref{lemF} this is equivalent to the corresponding result on the undeformed domains:
$$
h_T \kappa_i^2 \Vert p \Vert_{L^2(\Gammalin_T)}^2  \lesssim \Vert p \Vert_{L^2(T_i)}^2 \quad \text{for all}~ p \in \mathcal{P}^k(T).
$$
We transform the problem to the reference element $\hat{T} := \{ x \in \rr^d | \sum_{i=1}^d x_i \leq 1, x_i \geq 0, i=1,..,d \}$ with the corresponding affine linear transformation $\Phi_T : \hat{T} \rightarrow T$, such that $T = \Phi_T(\hat{T})$. 
We introduce $J_V := |\det(D \Phi_T)|$, $J_\Gamma := J_V \Vert (D \Phi_T)^{-T}  n_{\hat{\Gamma}} \Vert_2$, $\hat{\Gamma} = \Phi_T^{-1} (\Gammalin_T)$ and $\hat{T}_i := \Phi_T^{-1} (T_i)$ and then have  
$
 \int_{T_i} p^2 \, dx = J_V \int_{\hat{T}_i} (p \circ \Phi_T)^2 \, dx,
$ and with the same arguments as in \eqref{changemeas} one gets $
 \int_{\Gammalin_T} p^2 \, ds = J_\Gamma \int_{\hat{\Gamma}} (p \circ \Phi_T)^2 \, ds. 
$
Note that
\[
  \|D \Phi_T\|_2^{-1} \leq \Vert (D \Phi_T)^{-T}  n_{\hat{\Gamma}} \Vert_2 \leq \|D \Phi_T^{-1}\|_2.
\]
Due to the assumption of shape regularity we have $ \|D \Phi_T\|_2 \lesssim h_T$, $\|D \Phi_T^{-1}\|_2 \lesssim h_T^{-1}$ and thus we get  $J_V \sim J_\Gamma h_T$. Hence,  it suffices to show
\begin{equation}\label{investhat}
\kappa_i^2 \Vert p \Vert_{\hat{\Gamma}}^2  \leq c_{k,d} \Vert p \Vert_{\hat{T}_i}^2 \quad \text{for all}~\, p \in \mathcal{P}^k(\hat{T}),~i=1,2
\end{equation}
with a constant $c_{k,d}$ depending only on the polynomial degree $k$, the dimension $d$ and the constant $c_{\kappa}$. We now prove the estimate \eqref{investhat}.

 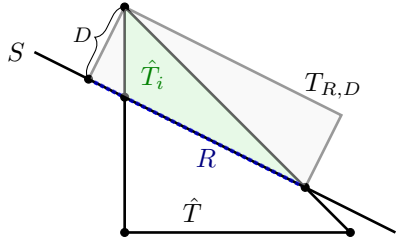
\begin{figure}
   \begin{center}
     \begin{tikzpicture}[scale=3.0]
       \coordinate (V0) at (0.0,0.0);
       \coordinate (V1) at (1.0,0.0);
       \coordinate (V2) at (0.0,1.0);
       \coordinate (V01) at (0.8,0.2);
       \coordinate (V20) at (0.0,0.6);
       \coordinate (Q1) at (-0.4,0.8);
       \coordinate (Q2) at (1.2,0.0);
       \coordinate (Q3) at (-0.16,0.68);
       \coordinate (P1) at (0.96,0.52);
       \draw[line width=1pt] (V0) -- (V1) -- (V2) --cycle;
       \draw[fill=green,opacity=0.12] (V01) -- (V20) -- (V2) --cycle;
       \draw[line width=1pt] (Q1) -- (Q2);
       \draw[line width=1.5pt,blue,densely dotted] (Q3) -- (V01);
       \filldraw (V0) circle (0.5pt);
       \filldraw (V1) circle (0.5pt);
       \filldraw (V2) circle (0.5pt);
       \filldraw (V01) circle (0.5pt);
       \filldraw (V20) circle (0.5pt);
       \filldraw (Q3) circle (0.5pt);
       \draw[fill=gray!12,line width=1pt,opacity=0.4] (Q3) -- (V2) -- (P1) -- (V01) --cycle;
       \node at (-0.4,0.8) [left] {$S$};
       \node at (0.36,0.33) {\color{blue!50!black}$R$};
       \node at (0.3,0.1) {$\hat{T}$};
       \node at (0.125,0.7) {\color{green!50!black}$\hat{T}_i$};
       \node at (0.925,0.65) {$T_{R,D}$};
       \draw [decorate,decoration={brace,amplitude=5pt},xshift=-9pt,yshift=4pt] (Q3) -- (V2) node [black,midway,above left,xshift=-2pt,yshift=-3pt] {\footnotesize $D$};
     \end{tikzpicture}
   \end{center}
   \caption{Sketch of geometries used in the proof of Lemma \ref{leminv}}
\label{proofsketch}
 \end{figure}

As $\hat{\Gamma}$ is planar $\hat{T}_i$ is a convex polytope.
Let $S$ be the hyperplane, with normal denoted by $n_S$, that contains $\hat{\Gamma}$ and $R = \{ x \in S~|~x + \alpha n_S \in \hat{T}_i, \alpha \in \rr\}$ be the normal projection of $\hat{T}_i$ into $S$, cf. Fig.~\ref{proofsketch}. 
 With $\I_i$ the subset of  those vertices of $\hat T$ that are vertices of $\hat{T}_i$ we set $D := \max_{x_V \in \I_i}\mathrm{dist}(S,x_V)$. We define the cylinder $T_{R,D} = \{ x \in \rr^d ~|~ x = y + \alpha n_S, y \in R, \alpha \in [0,D]\}$. Note that  $\hat{T}_i \subset T_{R,D}$ holds.
We can bound the volume of $T_{R,D}$ by $ |T_{R,D}| \leq |R| D \leq l^{d-1} D $, where $l$ is the maximum distance of two points in $R$ which can be bounded by the length of the longest edge of $\hat{T}$, i.e., $l \leq \sqrt{2}$. 
Combining this with  $|\hat{T}_i| \geq \frac{\kappa_i^2}{c_{\kappa}} |\hat{T}|$ from \eqref{eq:crit} we get
$$
\frac{\kappa_i^2}{c_{\kappa} d!} = \frac{\kappa_i^2}{c_{\kappa}} |\hat{T}| \leq |\hat{T}_i| \leq |T_{R,D}| \leq l^{d-1}  D ,
$$
and thus
$ D \geq c_d \frac{\kappa_i^2}{c_{\kappa}}$ with $c_d := (d!  2^{\frac{d-1}{2}})^{-1}$. From this it follows that there  exists a vertex $x_V$ of the reference simplex $\hat T$ such that  $x_V \in \hat{T}_i$ and $\mathrm{dist}(S,x_V) \geq  c_d \frac{\kappa_i^2}{c_{\kappa}}$. 

Note that $\mathrm{conv}(\hat{\Gamma},x_V) \subset \hat{T}_i$. Based on a decomposition of $\hat{\Gamma}$ into $(d-1)$-simplices $\{G_l\}_{l=1,..,L}$ we consider the (non-overlapping) decomposition of $\mathrm{conv}(\hat{\Gamma},x_V)$ into $d$-simplices  $T_l:=\mathrm{conv}(G_l,x_V)$, $l=1,..,L$. 
On $T_l$ we apply known inverse trace  estimates for polynomials, c.f. \cite{warburtonhesthaven03}:
$
  \Vert p \Vert_{L^2(G_l)}^2 \leq \frac{(k+d)(k+1)}{d} \frac{|G_l|}{|T_l|} \Vert p \Vert_{L^2(T_l)}^2
$.
With $|T_l| = \frac{D}{d} |G_l| \geq \frac{\kappa_i^2}{c_{\kappa}} \frac{c_d}{d} |G_l| $  we obtain
$\kappa_i^2 | G_l | / |T_l| \leq \frac{d}{c_d} c_\kappa$ and thus
$$
  \kappa_i^2 \Vert p \Vert_{L^2(G_l)}^2 \leq c_{\kappa} \frac{(k+d)(k+1)}{c_d} \Vert p \Vert_{L^2(T_l)}^2.
$$
By summing over $l=1,..,L$ this proves \eqref{investhat} with $c_{k,d} = c_{\kappa} d! 2^{\frac{d-1}{2}} (k+d)(k+1)$.
\end{proof}\\
\begin{remark} \rm
The condition \eqref{eq:crit} on the averaging operator is crucial for the estimate \eqref{inverse} to hold. Furthermore, the fact that the interface is piecewise planar w.r.t. the reference simplex is used in the analysis. 
We briefly comment on similar results from the literature. 
For $k=1$ and the weighting $\kappa_i =|T_i|/|T|$ the inverse estimate has been shown in \cite{hansbo2002unfitted}. 
In \cite{massjung12} the result in Lemma \ref{leminv} has been proven for the Heaviside choice $\kappa_i \in \{ 0,1 \}$ for the case of a higher order unfitted discontinuous Galerkin discretization for smooth interfaces in two dimensions. A variant of this method has been addressed in \cite{wuxiao10}. We are not aware of any literature in which such a result for higher order discretizations and dimension $d \geq 3$ has been derived.
\end{remark}
\\[1ex]
A large part of the analysis of the usual unfitted finite element methods can be carried over using the following result:
\begin{lemma} \label{lemcoercive}
For $\lambda $ sufficiently large the estimates
\begin{align}
  A_h(u,u) &\gtrsim \Vert u \Vert_h^2 && \hspace*{-2cm} \text{for all}~\, u \in V_{h,\Theta}^\Gamma, \label{eq:coerc} \\
  A_h(u,v) &\lesssim \Vert u \Vert_h \Vert v \Vert_h && \hspace*{-2cm} \text{for all}~\, u,v \in \Vreg + V_{h,\Theta}^\Gamma \label{eq:bound}
\end{align}
hold.
\end{lemma}
\begin{proof}
We have
\begin{equation} \label{r1}
  a_h(u,u) \sim |u|_1^2 \quad \text{for}~u \in  \Vreg + V_{h,\Theta}^\Gamma.
\end{equation}
Furthermore, 
\begin{equation} \label{r2}
 \begin{split}
 |N_h^s(u,v)|  & \lesssim \lambda \Vert \jump{u} \Vert_{\frac12,h,\Gamma_h} \Vert \jump{v} \Vert_{\frac12,h,\Gamma_h}, \quad u,v \in H^1(\Omega),\\  
 N_h^s(u,u)  & \sim \lambda \Vert \jump{u} \Vert_{\frac12,h,\Gamma_h}^2, \quad u \in H^1(\Omega).
\end{split} \end{equation}
For the Nitsche consistency term we have, for $u,v \in \Vreg + V_{h,\Theta}^\Gamma$:
\begin{equation} \label{r3}
  |N_h^c(u,v)| \lesssim  \Vert \average{\alpha \nabla u} \Vert_{-\frac12,h,\Gammalin}\Vert \jump{v} \Vert_{\frac12,h,\Gammalin}.
\end{equation}
The result in \eqref{eq:bound} follows from the definition of $A_h(\cdot,\cdot)$ and \eqref{r1}, \eqref{r2}, \eqref{r3}.  
Using \eqref{inverse} and the results in \eqref{r1}, \eqref{r2}, \eqref{r3} we get, for $u \in V_{h,\Theta}^\Gamma$,
\begin{align*}
  A_h(u,u) & \gtrsim \vert u \vert_{1}^2 + \lambda \Vert \spacejump{u} \Vert_{\frac12,h,\Gamma_h}^2 -  \vert u \vert_{1}  \, \Vert \spacejump{u} \Vert_{\frac12,h,\Gamma_h} \\
& \gtrsim  \vert u \vert_{1}^2 + \Vert \spacejump{u} \Vert_{\frac12,h,\Gamma_h}^2,
\end{align*}
provided $\lambda$ is chosen sufficiently large. Using \eqref{inverse} again we obtain the estimate
 \eqref{eq:coerc}.
\end{proof}
\ \\[1ex]
In the remainder we assume that $\lambda$ is taken sufficiently large such that the results in Lemma~\ref{lemcoercive} hold. Then the discrete problem \eqref{Nitsche1} has a unique solution. 

\subsection{The bijective mapping $\Phi_h= \Psi \circ \thetah^{-1}$} \label{sectmappsi}
From the properties derived in the sections~\ref{sec:globaltrafoPsi} and \ref{sec:globaltrafo} it follows that, for $h$ sufficiently small, the mapping $\Phi_h= \Psi \circ \thetah^{-1}$ is bijection on $\Omega$ and has the property $\Phi_h(\Gamma_h)=\Gamma$. In the remainder we assume that $h$ is sufficiently small such that $\Phi_h$ is a bijection. It has the smoothness property $\Phi_h \in C(\Omega)^d \cap C^{k+1}(\T)^d$.  In the following lemma we derive  further properties of $\Phi_h$ that will be needed in the error analysis.
\begin{lemma} \label{lem:phihbounds}
 The following holds:
  \begin{align}
    \|\thetah - \Psi\|_{\infty,\Omega}+ h\|D(\thetah - \Psi)\|_{\infty,\Omega} & \lesssim h^{k+1} \label{estt1} \\
     \|\Phi_h - \id\|_{\infty,\Omega}+ h\|D\Phi_h - I\|_{\infty,\Omega} & \lesssim h^{k+1}.\label{estt2} 
  \end{align}
\end{lemma}
\begin{proof}
  Using the definition $\Psi=\mathcal{E}\PsiGamma$, $\thetah=\mathcal{E}\thetahGamma$ and the results in Theorem~\ref{exthm} and Lemma~\ref{lem4} we get the result in \eqref{estt1} with
  \begin{align*}
\|\thetah - \Psi\|_{\infty,\Omega} & + h\|D(\thetah - \Psi)\|_{\infty,\Omega} \\ & \lesssim \|\thetahGamma - \PsiGamma\|_{\infty,\OGamma} + h \|D(\thetahGamma - \PsiGamma)\|_{\infty,\OGamma}  \\ & \qquad +\max_{F \in \F(\partial \OGamma)} \sum_{r=0}^{k+1} h^{r} \|D^r (\thetahGamma - \PsiGamma)\|_{\infty,F} +\max_{x_i \in \V(\partial \OGamma)}|(\thetahGamma-\PsiGamma)(x_i)| \\
  & \lesssim \sum_{r=0}^{k+1} h^r \max_{T \in \TGamma}\Vert D^r( \thetahGamma - \Psi^\Gamma) \Vert_{\infty,T} 
\lesssim h^{k+1}.
 \end{align*}
Note that $\Phi_h-\id=(\Psi-\thetah)\thetah^{-1}$. From  the results in  \eqref{bpsi1} and \eqref{estt1} it follows that $\|\thetah^{-1}\|_{\infty,\Omega} \lesssim 1$, $\|D\thetah^{-1}\|_{\infty,\Omega} \lesssim 1$. Hence, the estimate in \eqref{estt2} follows from the one in \eqref{estt1}.
\end{proof}
\\[1ex]
\begin{remark} \rm With similar arguments as used in the proof above one can also derive the bound
\[
 \max_{T \in \T} \Vert D^l (\Phi_h - \id) \Vert_{\infty,T} = \max_{T \in \T} \Vert D^l \Phi_h  \Vert_{\infty,T}  \lesssim h^{k+1-l}, \qquad T \in \T, ~ l=2,..,k+1,
\]
but we do not need this in the error analysis.
\end{remark}

\subsection{Strang lemma} \label{sectStrang}
In this section we derive a  Strang lemma in which the discretization error is related to approximation and geometry errors. We use the homeomorphism $\Phi_h: \Omega \to \Omega $ with the property $\Phi_h(\Gamma_h)=\Gamma$. 

We define $\Vregphi :=H^1(\Omega)\cap  H^2(\Omega_1 \cup \Omega_2) $, cf. \eqref{defVreg}.
Related to  $\Phi_h$ we define 
\[ V_{h,\Phi}^\Gamma := \{ v \circ \Phi_h^{-1}, v \in V_{h,\Theta}^\Gamma\} \subset H^1(\Omega_1 \cup \Omega_2),
\]
and the linear and bilinear forms
\begin{equation}
A(u,v) := a(u,v) + N(u,v), \quad f(v) := \sum_{i=1,2} \int_{\Omega_i} f_i v \, dx, \quad u,v \in V_{h,\Phi}^\Gamma + \Vregphi,
\end{equation}
with the bilinear forms $a(\cdot,\cdot)$, $N(\cdot,\cdot)$ as in \eqref{eq:blfs} with $\Omega_{i,h}$ and $\Gamma_h$ replaced with $\Omega_i$ and $\Gamma$, respectively. In the following lemma a consistency result is given.
\begin{lemma} \label{lemconsis}
Let $u\in \Vregphi$ be a solution of \eqref{eq:ellmodel}.  The following holds:
\begin{equation} \label{eq:cons1}
  A(u,v) = f(v) \quad \text{for all}~ \, v \in V_{h,\Phi}^\Gamma + \Vregphi.
\end{equation}
\end{lemma}
\begin{proof}
Take $v \in V_{h,\Phi}^\Gamma + \Vregphi$. Since $\spacejump{u}_{|\Gamma}=0$ a.e. on $\Gamma$, we get $N^s(u, v)= N^c(v,u)=0$. Thus we obtain
\begin{align*}
  A(u,v)- f(v) &  = a(u,v)+ N^c(u, v) - \sum_{i=1,2} \int_{\Omega_i} f_i v \, dx \\
& = \sum_{i=1}^2 \int_{\Omega_i} \alpha_i \nabla u \cdot \nabla v - f_i v \, dx +\int_{\Gamma}  \average{-\alpha \nabla  u \cdot n} \spacejump{ v} \, ds\\
& = \sum_{i=1}^2 \int_{\Omega_i} \alpha_i \nabla u \cdot \nabla v - f v_i \, dx+ \int_{\Gamma} \spacejump{-\alpha \nabla  u \cdot n v} \, ds,
\end{align*}
where in the last equality we used the flux continuity in \eqref{eq:ellmodel2}. Applying partial integration on $\Omega_i$ and using \eqref{eq:ellmodel1} results in
\[
  A(u, v)- f(v)=\sum_{i=1}^2 \int_{\Omega_i}- \div (\alpha_i \nabla u) v - f_i v \, dx =0,
\]
which completes the proof.
\end{proof}
\ \\
\begin{lemma}\label{lem:strang}
Let $u\in \Vregphi$ be a solution of \eqref{eq:ellmodel} and $u_h \in V_{h,\Theta}^\Gamma$ the solution of \eqref{Nitsche1}. The following holds:
\begin{equation}\label{eq:strang}
\begin{split}
\Vert u \circ \Phi_h - u_h \Vert_h \lesssim &  \inf_{v_h \in V_{h,\Theta}^\Gamma} \Vert u\circ \Phi_h - v_h \Vert_h \\ 
 & + \sup_{w_h \in V_{h,\Theta}^\Gamma} \frac{ | f_h(w_h) - f(w_h \circ \Phi_h^{-1}) | }{\Vert w_h \Vert_h} \\
 & + \sup_{w_h \in V_{h,\Theta}^\Gamma} \frac{ | A_h(u\circ \Phi_h,w_h) - A(u,w_h \circ \Phi_h^{-1}) | }{\Vert w_h \Vert_h}.
\end{split}
\end{equation}
\end{lemma}
\begin{proof}
The proof is along the same lines as in the well-known Strang Lemma. We use the notation $\tilde u =u \circ \Phi_h$ and start with the triangle inequality, where we use an arbitrary $v_h \in V_{h,\Theta}^\Gamma$:
\begin{equation*}
\Vert \tilde u - u_h \Vert_h \leq \Vert \tilde u - v_h \Vert_h  + \Vert v_h - u_h \Vert_h.
\end{equation*}
With $V_{h,\Theta}^\Gamma$-coercivity, cf. \eqref{eq:coerc}, we have
\begin{equation} \label{eq:strang2}
\begin{split}
\Vert u_h - v_h \Vert_h^2 & \lesssim A_h(u_h - v_h,u_h - v_h) \stackrel{w_h=u_h-v_h}{=} A_h(u_h-v_h,w_h) \\
& = \left| A_h(\tilde u-v_h,w_h) \right| + \left| A_h(\tilde u,w_h) - f_h(w_h) \right|.
\end{split}
\end{equation}
Using continuity, \eqref{eq:bound}, and dividing by $\Vert w_h \Vert_h = \Vert u_h - v_h \Vert_h$ results in
\begin{equation*}
\Vert \tilde u - u_h \Vert_h \lesssim \inf_{v_h \in V_{h,\Theta}^\Gamma} \Vert \tilde u - v_h \Vert_h + \sup_{w_h \in V_{h,\Theta}^\Gamma} \frac{ | A_h(\tilde u,w_h) - f_h(w_h) | }{\Vert w_h \Vert_h}.
\end{equation*}
Using the consistency property of Lemma~\ref{lemconsis} yields
\begin{align*}
& |A_h(\tilde u,w_h) - f_h( w_h)| \\
& \qquad = |A_h(\tilde u,w_h) - f_h( w_h) - \underbrace{\left(A(u,w_h \circ \Phi_h^{-1}) - f(w_h\circ \Phi_h^{-1})\right)}_{=0}| \\ 
& \qquad \leq |A_h( \tilde u, w_h) - A(u, w_h \circ \Phi_h^{-1})| + |f_h( w_h) - f(w_h \circ \Phi_h^{-1})|
\end{align*}
which completes the proof.
\end{proof}
\\
\begin{remark} \rm
\label{rem:quad}
In practice we will commit an additional variational crime due to the numerical integration of integrands as in \eqref{atrans} which are in general not polynomial. We briefly discuss how to incorporate the related additional error terms in the error analysis without going into all the technical details. The procedure is essentially the same as in Sect. 25-29 in \cite{CiarletHandbook}.
We consider $A_h^q$ and $f_h^q$ the bilinear and linear forms which approximate $A_h$ and $f_h$ with numerical integration. 
The solution obtained by quadrature is denoted as $u_h^q \in V_{h,\Theta}^\Gamma$ which satisfies $A_h^q(u_h^q,v_h) = f_h^q(v_h)$ for all $v_h \in V_{h,\Theta}^\Gamma$.
For the numerical integration we consider quadrature rules that have a certain exactness degree  on the (cut) elements in  the reference geometries $\Omega^{\text{lin}}_i$ and $\Gamma^{\text{lin}}$. For the integral in \eqref{atrans} this would be a quadrature of exactness degree $2k-2$. 
The additional terms $D\thetah^{-T}$ and $\mathrm{det}(D\thetah)$ stemming from the transformation are close to $\id$ and $1$ respectively, cf. the estimates in Lemma \ref{lemF}. As a consequence we have that $A_h(u,u) \sim A_h^q(u,u)$ on $V_{h,\Theta}^\Gamma$ and thus we get coercivity of $A_h^q(\cdot,\cdot)$ w.r.t. $\|\cdot\|_h$ on $V_{h,\Theta}^\Gamma$.
The arguments in the proof of the Strang Lemma~\ref{lem:strang} apply with $A_h(\cdot,\cdot)$ replaced by  $ A_h^q(\cdot,\cdot)$. Hence  we have to bound consistency terms as in \eqref{eq:strang} with $A_h(\cdot,\cdot)$, $f_h$ replaced by $A_h^q(\cdot,\cdot)$ and $f_h^q$, respectively. This can be done by combining the techniques used in the proof of Lemma~\ref{lemconsist} below with the ones from \cite{CiarletHandbook}. 
\end{remark}

\subsection{Consistency errors} \label{sec:integrals} 
We derive bounds for the consistency error terms on the right-hand side in the Strang estimate \eqref{eq:strang}.
\begin{lemma} \label{lemconsist}
Let $u \in  \Vregphi$ be a solution of \eqref{eq:ellmodel}. We assume $f \in H^{1,\infty}(\Omega_1 \cup \Omega_2)$ and the data extension source term $f_h$ in \eqref{eq:lfs} satisfies $\|f_h\|_{H^{1,\infty}(\Omega_{1,h}\cup \Omega_{2,h})} \lesssim \|f\|_{H^{1,\infty}(\Omega_1 \cup \Omega_2)}$. The following estimates hold for $w_h \in V_{h,\Theta}^\Gamma$:
\begin{subequations}
\begin{align}
 |A(u,w_h\circ \Phi_h^{-1}) - A_h(u\circ \Phi_h,w_h)| & \lesssim h^k \Vert u \Vert_{H^2(\Omega_1 \cup \Omega_2)} \Vert w_h \Vert_h \label{p1}
\\
 |f(w_h\circ \Phi_h^{-1})-f_h(w_h)| & \lesssim h^{k} \|f\|_{H^{1,\infty}(\Omega_1 \cup \Omega_2)} \| w_h \Vert_h \label{p2}.
\end{align}
\end{subequations}
\end{lemma}
\begin{proof} The proof, which is elementary, is given in the Appendix.
\end{proof}

\subsection{Approximation error}\label{approx}
In this section we derive a bound for the approximation error, i.e., the first term on the right-hand side in the Strang estimate \eqref{eq:strang}.
For this we use the curved finite element space $V_{h,\Psi}=\{\, v_h \circ \Psi_h^{-1}~|~v_h \in V_h^k\,\}$ introduced in \eqref{FEcurvedPsi} and the corresponding optimal interpolation operator $\Pi_h$, cf.~\eqref{errorBernardi}. We also use the corresponding unfitted finite element space $V_{h,\Psi}^\Gamma:=\{\, v_h \circ \Psi_h^{-1}~|~v_h \in V_h^\Gamma\,\}$, cf.\eqref{transfspace}.

\begin{lemma} \label{lemapprox} For
  $u \in H^{k+1}(\Omega_1 \cup \Omega_2)$
  the following holds:
\begin{equation}\label{approxerrro}
  \inf_{v_h \in V_{h,\Theta}^\Gamma} \Vert u\circ \Phi_h - v_h \Vert_h =\inf_{v_h \in V_{h,\Psi}^\Gamma} \Vert (u- v_h)\circ \Phi_h \Vert_h  \lesssim h^k  \|u\|_{H^{k+1}(\Omega_1 \cup \Omega_2)}.
\end{equation}
\end{lemma}
\vspace*{-0.5cm}
\begin{proof}
The identity in \eqref{approxerrro} follows from the definitions of the spaces $V_{h,\Theta}^\Gamma$, $V_{h,\Psi}^\Gamma$. 
The analysis is along the same lines as known in the literature, e.g. \cite{hansbo2002unfitted,reusken08}. 
Let $\Pi_h$ be the  interpolation operator in $V_{h,\Psi}$ and $R_i$ the restriction operator $R_iv:=v_{|\Omega_{i}}$.
We use the bounded linear extension operators $E_i:\, H^{k+1}(\Omega_i) \to H^{k+1}(\Omega)$, cf., for example, Theorem II.3.3 in \cite{GaldibookNS} and define $u_i^e := E_i u_i,~i=1,2$.
We define $v_h := \sum_{i=1}^2 R_i \Pi_h u_i^e  \in V_{h,\Psi}^\Gamma$. 
Note that due to  \eqref{errorBernardi} we have a optimal interpolation error bounds for $u_i^e  \in H^{k+1}(\Omega)$.
We get:
\begin{align*}
  |(u - v_h)\circ \Phi_h|_1^2& \lesssim  \sum_{i=1}^2 \alpha_i \|\nabla (u_i  - R_i \Pi_h u_i^e )\|_{L^2(\Omega_{i})}^2 \\
 & \lesssim \sum_{i=1}^2 \|\nabla ( u_i^e  - \Pi_h u_i^e)\|_{L^2(\Omega)}^2
 \lesssim h^{2k} \sum_{i=1}^2 \| u_i^e \|_{H^{k+1}(\Omega)}^2.
\end{align*}
This yields the desired bound for the $|\cdot|_1$ part of the norm $\|\cdot\|_h$. As in \cite{hansbo2002unfitted}, we obtain
\begin{align*}
  \Vert \jump{(u- v_h)\circ \Phi_h} \Vert_{\frac12,h,\Gamma_h} &= \left( \bar \alpha / h \right)^\frac12 \|\jump{(u - v_h)\circ \Phi_h}\|_{L^2(\Gamma_h)}  \lesssim h^{- \frac12} \sum_{i=1}^2 \| u_i^e  - \Pi_h u_i^e \|_{L^2(\Gamma)} \\
& \lesssim \sum_{i=1}^2\Big( h^{-1} \| u_i^e - \Pi_h u_i^e \|_{L^2(\Omega)} +\| u_i^e  - \Pi_h u_i^e \|_{H^1(\Omega)}\Big) \\
& \lesssim h^k \sum_{i=1}^2 \| u_i^e \|_{H^{k+1}(\Omega)}  \lesssim h^k \sum_{i=1}^2 \| u_i \|_{H^{k+1}(\Omega_i)}. 
\end{align*}
Using very similar arguments, cf.~ \cite{hansbo2002unfitted}, the same bound can be derived for the  term $\Vert \average{(u- v_h)\circ \Phi_h}\Vert_{-\frac12,h,\Gamma_h}$. Thus we get $\|(u-v_h)\circ \Phi_h\|_h \lesssim h^k  \| u\|_{H^{k+1}(\Omega_1 \cup \Omega_2)} $ which completes the proof. 
\end{proof}
\subsection{Discretization error bounds} \label{sectH1}
Based on the Strang estimate and the results derived in the subsections above we immediately obtain the following theorem, which is the main result of this paper.

\begin{theorem} \label{mainthm}
  Let $u$ be the solution of \eqref{eq:ellmodel} and $u_h \in V_{h,\Theta}^\Gamma$  the solution of \eqref{Nitsche1}.
We assume that  $u\in H^{k+1}(\Omega_1 \cup \Omega_2)$ and the data extension $f_h$ satisfies the condition in Lemma~\ref{lemconsist}.  Then the following holds:
\begin{equation}\label{main1}
\Vert u\circ \Phi_h - u_h \Vert_h \lesssim h^k( \| u \|_{H^{k+1}(\Omega_1 \cup \Omega_2)} + \|f\|_{H^{1,\infty}(\Omega_1\cup\Omega_2)}).
\end{equation}

\end{theorem}
\begin{proof}
The result directly follows from the Strang Lemma~\ref{lem:strang} and the bounds derived in Lemma~\ref{lemconsist} and Lemma~\ref{lemapprox}. 
\end{proof}

\section{Discussion and outlook} \label{sectoutlook}
We presented a rigorous error analysis of a high order unfitted finite element method applied to a model interface problem. The key component in the method is a parametric mapping $\thetah$, which transforms the piecewise linear planar interface approximation $\Gammalin$ to a higher order approximation $\Gamma_h$ of the exact interface $\Gamma$. This mapping is based on a new local level set based mesh transformation combined with an extension technique known from isoparametric finite elements. The corresponding isoparametric
 unfitted finite element space is used for discretization and combined with a standard Nitsche technique for enforcing continuity of the solution across the interface in a weak sense. The discretization error analysis is based on a Strang Lemma. 
For handling the geometric error a suitable bijective mapping $\Phi_h$ on $\Omega$ is constructed which maps the numerical interface $\Gamma_h$  to the exact interface $\Gamma$ and is sufficiently close to the identity. The construction of this mapping is based on the same approach as used in the construction of $\thetah$. The main results of the paper are the optimal high order ($H^1$-norm) error bounds given in Section~\ref{sectH1}. 

The derivation of an optimal order $L^2$-norm error bound, based on a duality argument, will be presented in a forthcoming paper. For this we need an improvement of  the consistency error bound in \eqref{p1}, which is  of order  $\mathcal{O}(h^k)$, to a bound of order $\mathcal{O}(h^{k+1})$. 

The methodology presented in this paper, especially the parametric mapping and the error analysis of the geometric errors, can also be applied in other settings with unfitted finite element methods, for example, fictitious domain methods \cite{L_ARXIV_2016}, unfitted FEM for  Stokes interface problems \cite{LPWL_PAMM_2016} and trace finite element methods for surface PDEs \cite{GLR16}.

\section{Appendix} \ \\
\subsection{Proof of Lemma~\ref{propertiesdh}} \label{proof:propertiesdh}
First we prove existence of a unique solution $d_h(x)$ of \eqref{eq:psihmap}. 
For $|\alpha| \leq \alpha_0 h$, with $\alpha_0 >0$ fixed, we introduce the polynomial
\[
  p(\alpha):=  \mathcal{E}_T \phi_h(x + \alpha G_h(x)) - \hat \phi_h (x)
\]
for a fixed $x\in  T \in \TGamma$. From \eqref{err1} it follows that 
\begin{equation} \label{estt} 
| \hat \phi_h (x) - \phi(x)| \lesssim h^2.
\end{equation}
 Furthermore, with $y:=x +\alpha G_h(x)$, we have $\|x-y\|_2 \lesssim h$, and using a Taylor expansion and \eqref{err2} we obtain (for $\xi \in (0,1)$)
\begin{equation} \label{Taylor}
 \begin{split}
  \big|  (\mathcal{E}_T \phi_{h})(y)-\phi(y)\big| = & \ \big| \sum_{m=0}^k \frac{1}{m!} D^m( \mathcal{E}_T \phi_{h}-\phi)(x)(y-x, \ldots,y-x) \\ & \ + \frac{1}{(k+1)!}D^{k+1}\phi(x+\xi(y-x))(y-x,\ldots,y-x)\big|
\\
 \lesssim & \sum_{m=0}^k |\phi_h -\phi|_{m, \infty,T} \|y-x\|_2^m + \|y-x\|_2^{k+1} \lesssim h^{k+1}.
\end{split}
\end{equation} 
Thus we get $p(\alpha)= \phi(x+ \alpha G_h(x))- \phi(x) +\mathcal{O}(h^2)$, where the constant in $\mathcal{O}(\cdot)$ is independent of $x$. Using  Lemma~\ref{lem2} and \eqref{eq1} we obtain 
\begin{equation} \label{lk} 
p(\alpha)=  \alpha \|\nabla \phi(x)\|_2^2 + \mathcal{O}(h^2).
\end{equation}
 Hence, there exists $h_0 >0$ such that for all $0< h\leq h_0$ and all $x\in  T \in \TGamma$ the equation $p(\alpha)=0$ has a unique solution $\alpha =:d_h(x)$ in $[-\alpha_0 h,\alpha_0 h]$. The result in \eqref{resd4} follows from \eqref{lk}. The smoothness property $d_h \in C^\infty(T)$, $T \in \TGamma$, follows from the fact that $G_h$ and $\hat \phi_h$ are polynomials on $T$ and $\mathcal{E}_T \phi_h$ is a global polynomial. \\
We note that for each vertex $x_i$ of $T \in \T$ we have $\mathcal{E}_T \phi_h(x_i) = \phi_h(x_i) = \hphi(x_i)$ and it follows that $d_h(x_i)=0$ solves \eqref{eq:psihmap} and hence \eqref{resd1} holds. 
We finally consider \eqref{resd6}. We differentiate the relation \eqref{eq:psihmap} on $T$. We skip the argument $x$ in the notation and use $y_h:=x+d_h(x) G_h(x)$:
\begin{equation} \label{kjj} 
 \nabla(\mathcal{E}_T\phi_h)(y_h) + \nabla d_h \nabla (\mathcal{E}_T\phi_h)(y_h)^TG_h + d_h DG_h^T \nabla (\mathcal{E}_T\phi_h)(y_h) = \nabla \hat \phi_h.
\end{equation}
Using a Taylor expansion as in \eqref{Taylor} we get $
  |\nabla(\mathcal{E}_T\phi_h)(y_h)-\nabla \phi(y_h)| \lesssim h^k $. Using \eqref{resd4}, \eqref{eqlem2} and rearranging terms in \eqref{kjj} we get (for $k \geq 2$):
\[
  \nabla d_h \big(\|\nabla \phi\|_2^2 + \mathcal{O}(h^2)\big)= \nabla \hat \phi_h - \nabla(\mathcal{E}_T\phi_h)(y_h)-d_h DG_h^T \nabla (\mathcal{E}_T\phi_h)(y_h).
\]
From the estimates derived above and the smoothness assumption on $\phi$ it easily follows that the second and third term on the right-hand side are uniformly bounded by a constant and by $ch^2$, respectively. 
We further have $\nabla \hphi = \nabla \phi + \nabla \hphi - \nabla \phi$ with 
$\Vert \nabla \hphi - \nabla \phi \Vert_{\infty} \lesssim h$ and $\Vert \nabla \phi \Vert_{\infty} \lesssim 1$. Combining these results completes the proof. 

\subsection{Proof of Lemma~\ref{lem3}} \label{proof:lem3}
Recall that $\Psi^\Gamma(x)= x+ d(x) G(x)$, $\Psi_h^\Gamma(x)=x+d_h(x)G_h(x)$. 
 We have $\|d\|_{\infty,\OGamma} \lesssim h^2,~ \max_{T \in \TGamma} \|d_h\|_{\infty,T} \lesssim h^2$, cf.~\eqref{resd4a}, \eqref{resd4}. Take $x \in T \in \TGamma$.
 We start with the triangle inequality
\begin{equation} \label{ttt} \begin{split}
  \|(\Psi^\Gamma - \Psi_h^\Gamma)(x)\|_2 & = \| d(x) G(x)- d_h(x)G_h(x)\|_2 \\
 & \leq |d(x)-d_h(x)| \|G\|_{\infty,\OGamma} + |d_h(x)| \, \|G_h(x)-G(x)\|_2.
\end{split}
\end{equation}
The second term on the right-hand side is uniformly  bounded by $h^{k+2}$ due to the result in  Lemma~\ref{lem2}.
We have by construction
\begin{equation}\label{phiphih}
  \phi(x+ d(x) G(x))= \hat\phi_h(x)=(\mathcal{E}_T\phi_{h}) (x+ d_h(x) G_h(x)), ~~x \in  T \in \TGamma.
\end{equation}
Using this and \eqref{eq1} we get
\begin{align*}
 |d(x)-d_h(x)| & \sim \big| \phi(x+ d(x) G(x)) -\phi(x+ d_h(x) G(x))\big| \\ & =\big| (\mathcal{E}_T \phi_{h})(x+ d_h(x) G_h(x))-\phi(x+ d_h(x) G(x))\big| \\
& \leq \big|  (\mathcal{E}_T \phi_{h}) (x+ d_h(x) G_h(x))-\phi(x+ d_h(x)  G_h(x))\big| \\ & \quad +\big|\phi(x+ d_h(x) G_h(x))- \phi(x+ d_h(x) G(x))\big|.
\end{align*}
Using the estimate in \eqref{Taylor} it follows that the first term on the right-hand side is uniformly bounded by $h^{k+1}$.
The second term on the right-hand side is uniformly bounded by
\[
  \big|\phi(x+ d_h(x)  G_h(x))- \phi(x+ d_h(x) G(x))\big| \lesssim |d_h(x)| \, \|G_h(x)- G(x)\|_2 \lesssim h^{k+2}.
\]
Collecting these results we obtain
\begin{equation} \label{NN}
 \max_{T \in \TGamma} \|d-d_h\|_{\infty,T} \lesssim h^{k+1}, 
\end{equation}
which yields the desired bound for the first term on the right-hand side in \eqref{ttt}. Hence the bound for the first term in \eqref{estPsi} is proven.
 It remains to bound the derivatives. 
We take  the derivative in the equation \eqref{phiphih}. To simpifly the notation we drop the argument $x$ and set $y:=x+d(x)G(x)$, $y_h:=x+d_h(x)G_h(x)$. This yields
\begin{equation} \label{kj} \begin{split}
& \nabla \phi(y) + \nabla d  \nabla \phi(y)^T \nabla \phi  + d DG^T \nabla \phi(y) \\ & = \nabla(\mathcal{E}_T\phi_h)(y_h) + \nabla d_h \nabla (\mathcal{E}_T\phi_h)(y_h)^T G_h + d_h DG_h^T \nabla (\mathcal{E}_T\phi_h)(y_h).
\end{split}
\end{equation}
Using a Taylor expansion as in \eqref{Taylor} we get
\begin{equation} \label{Tay}
  \|\nabla(\mathcal{E}_T\phi_h)(y_h)-\nabla \phi(y_h)\|_2 \lesssim h^k.
\end{equation}
Furthermore, $\|\nabla \phi(y)- \nabla \phi(x)\|_2 \lesssim h^2$, $ \|\nabla \phi(y_h)- \nabla \phi(x)\|_2 \lesssim h^2$ and $\|G_h(x)-\nabla \phi(x)\|_2 \lesssim h$.
Using these results and \eqref{eqlem2} and rearranging terms in \eqref{kj} we get
\begin{align*}
  & \big( \nabla d - \nabla d_h\big)\big(\|\nabla \phi\|_2^2 + \mathcal{O}(h)\big) \\ &  = \big( \nabla (\mathcal{E}_T\phi_h)(y_h)- \nabla \phi(y)\big) +\big( d_h DG_h^T \nabla(\mathcal{E}_T\phi_h)(y_h) - d DG^T \nabla \phi(y) \big) .
\end{align*}
Using  \eqref{eqlem2} and \eqref{NN} we get $|y-y_h| \lesssim h^{k+1}$. Using this and \eqref{Tay} we obtain for the first term on the right-hand side the uniform bound $c h^k$.  For the second term we obtain the same uniform bound  by using \eqref{Tay}, \eqref{NN}, \eqref{eqlem2} and $\max_{T \in \T}\|d_h\|_{\infty,T} \lesssim h^2$, $\|d\|_{\infty,\OGamma} \lesssim h^2$. This yields
\begin{equation} \label{re1}
 \max_{T \in \TGamma} \|\nabla d - \nabla d_h\|_{\infty,T} \lesssim h^k.
\end{equation}

Using this, \eqref{NN} and the results in  Lemma~\ref{lem2} and Lemma~\ref{propertiesdh} we obtain the estimates
\begin{align*}
 & \Vert D  (\Psi^\Gamma - \Psi_h^\Gamma) \Vert_{\infty,T}  = \Vert D (d G - d_h G_h) \Vert_{\infty,T} \\
& \lesssim \Vert \nabla (d-d_h) \Vert_{\infty, T} \Vert G \Vert_{\infty,T} + \Vert \nabla d_h \Vert_{\infty, T} \Vert G - G_h \Vert_{\infty, T} \\
 & + \|d-d_h\|_{\infty, T} \Vert D G \Vert_{\infty,T} + \|d_h\|_{\infty, T} \Vert D(G-G_h) \Vert_{\infty, T}  \lesssim h^k,
\end{align*}
which are uniform in $T \in \TGamma$. 

\subsection{Proof of Lemma~\ref{lem:lenoirext}} \label{proof:lenoirext}
Recall that $ \mathcal{E}^{F\rightarrow T} w := \mathcal{E}^{\hat F\rightarrow \hat T} (w \circ \hat \Phi_F^{-1}) \circ \Phi_T^{-1}$, $\hat{w} = w \circ \hat \Phi_F^{-1}$. We transform from $T$,$F$ to the corresponding reference element and face, and use 
$$
  \Vert D^n \mathcal{E}^{F \rightarrow T} w \Vert_{\infty, T} \lesssim
h^{-n}  \Vert D^n \mathcal{E}^{\hat F \rightarrow \hat T}  \hat w \Vert_{\infty, \hat{T}}, \quad 
  \Vert D^r \hat{w} \Vert_{\infty, \hat{F}} \lesssim
  h^r \Vert D^r w \Vert_{\infty, F}.
  $$
From this it follows that it suffices to prove
\begin{equation} \label{transfeq}
 \Vert D^n \mathcal{E}^{\hat F\rightarrow \hat T} \hat{w}_T \Vert_{\infty,\hat{T}} \lesssim \sum_{r=n}^{k+1} \Vert D^r \hat{w} \Vert_{\infty, \hat{F}}, \quad \forall \hat w \in C_0^{k+1}(\hat F), ~n=0,\ldots,k+1.
\end{equation}
Recall that
\begin{equation} \label{defE}
  \mathcal{E}^{\hat{F}\rightarrow \hat{T}} \hat w :=\omega^{k+1} A_k^\ast(\hat w)\circ Z+ \sum_{l=2}^k \omega^l  A_l( \hat{w}) \circ Z,
\end{equation} 
with $A_l^\ast := \id - \Lambda_l, \quad  A_l := \Lambda_l - \Lambda_{l-1} = - A_{l}^\ast + A_{l-1}^\ast$, $\Lambda_l : C(\hat{F}) \rightarrow \mathcal{P}^l(\hat{F})$ the interpolation operator, and $\omega$,$Z$ as in \eqref{defomega}.  
The proof uses  techniques also used in \cite{bernardi1989optimal}. 
 We apply a standard Bramble-Hilbert argument, which yields
 \begin{equation}\label{alstar}
    \Vert D^r (A_l^\ast \hat{w}) \Vert_{\infty,\hat{F}}  \lesssim  \Vert D^{l+1} \hat{w} \Vert_{\infty,\hat{F}}, \quad 0\leq r \leq l+1 \leq k+1,
  \end{equation}
and thus, by a triangle inequality,
\begin{equation}\label{lstar}
    \Vert D^r (A_l \hat{w}) \Vert_{\infty,\hat{F}}  \lesssim  \Vert D^{l+1} \hat{w} \Vert_{\infty,\hat{F}} +  \Vert D^{l} \hat{w} \Vert_{\infty,\hat{F}}, \quad 0\leq r \leq l+1 \leq k+1,~~l \geq 2.
  \end{equation}
  For $Z$ we have $\Vert D^m Z \Vert_{\infty,\hat{T}} \lesssim \omega^{-m},~m \geq 1$, cf. \cite[Lemma 6.1]{bernardi1989optimal}.
  With a higher order multivariate chain rule, cf. \cite{ciarlet1972interpolation}, we further have
\begin{align}
  \Vert D^m A_l^\ast(\hat{w}) \circ Z \Vert_{\infty,\hat{T}}
    & \lesssim
    \sum_{r=1}^m \underbrace{ \Vert D^r A_l^\ast(\hat{w}) \Vert_{\infty,\hat{F}}}_{ \lesssim \Vert D^{l+1} \hat{w} \Vert_{\infty, \hat{F}}} \sum_{\alpha \in \mathcal{J}(r,m)} \underbrace{\prod_{q=1}^m \Vert D^q Z \Vert^{\alpha_q} }_{\omega^{-m}} \\
  & \lesssim \Vert D^{l+1}\hat w \Vert_{\infty,\hat{F}} \, \omega^{-m}, \quad 1 \leq m \leq l+1\leq k+1, \label{kll} \\
     \text{ where } \quad \mathcal{J}(r,m) & := \{ \alpha = (\alpha_1,..,\alpha_r) \mid \sum_{q=1}^m \alpha_q = r, \sum_{q=1}^m q \alpha_q = m\}, \nonumber    
\end{align}
and for $m=0$ 
$$
  \Vert A_l^\ast(\hat{w}) \circ Z \Vert_{\infty,\hat{T}}
 =   \Vert A_l^\ast(\hat{w}) \Vert_{\infty,\hat{F}}
 \lesssim   \Vert D^{l+1} \hat{w} \Vert_{\infty,\hat{F}}.
$$
Similarly, using $D^mA_l(\hat w)=0$ for $m >l$, 
\[
 \Vert D^m A_l (\hat{w}) \circ Z \Vert_{\infty,\hat{T}} \lesssim \left( \Vert D^{l+1} \hat{w} \Vert_{\infty, \hat{F}} + \Vert D^{l} \hat{w} \Vert_{\infty, \hat{F}} \right) \omega^{-m},~ 0 \leq m, l\leq k+1, ~l \geq 2. 
\]
Using a Bramble-Hilbert argument, Leibniz formula and the bound derived in \eqref{kll} we get, for $n \leq k+1$,
\begin{align*}
 \|D^n(\omega^{k+1} A_k^\ast(\hat w)\circ Z)\|_{\infty,\hat F} 
 & \lesssim  \sum_{r=0}^{n} \omega^{k+1-r} \underbrace{|D \omega|^r}_{\lesssim 1} \Vert D^{n-r} (A_k^\ast(\hat w \circ Z)) \Vert_{\infty,\hat T} \\
& \lesssim  \sum_{r=0}^{k+1} \omega^{k+1-r}  \|D^{k+1}\hat w\|_{\infty,\hat F}\,\omega^{-(n-r)} \lesssim \|D^{k+1}\hat w\|_{\infty,\hat F},
\end{align*}
which yields the desired bound for the first term on the right hand-side in \eqref{defE}. We now consider one term $\omega^l A_l(\hat w)\circ Z$, $2 \leq l \leq k$, from the second term on the right hand-side in \eqref{defE}. This is a polynomial of degree $l$, hence $D^n(\omega^l A_l(\hat w)\circ Z) =0$ for $n >l$. We have to consider only $n \leq l \leq k$. 
With similar arguments as above we get:
\begin{align}
  \Vert D^n (\omega^l A_l(\hat w)\circ Z) \Vert_{\infty,\hat T} &
\lesssim \sum_{r=0}^n \omega^{l-r} |D \omega|^r \Vert D^{n-r} (A_l(\hat w \circ Z)) \Vert_{\infty,\hat T} \\ &\lesssim \sum_{r=0}^{k+1} \underbrace{\omega^{l-r} \omega^{-(n-r)}}_{\lesssim 1~~ (l \geq n)} ( \Vert D^l \hat w \Vert_{\infty,\hat F} +  \Vert D^{l+1} \hat w \Vert_{\infty,\hat F} ).
\end{align}
Summing over $l=2, \ldots, k$ completes the proof.

\subsection{Proof of Lemma~\ref{lemconsist}} \label{proof:lemconsist}
 We first consider the terms $a_h(\cdot,\cdot)$, $a(\cdot,\cdot)$ in $A_h(\cdot,\cdot)$ and $A(\cdot,\cdot)$, respectively, 
 split the bilinear forms into its contributions from the two subdomains, cf.~\eqref{atrans}, for instance $a_h(\cdot,\cdot) = a_h^1(\cdot,\cdot) + a_h^2(\cdot,\cdot)$, and consider one part $a_h^i(\cdot,\cdot)$, $i=1,2$.
We introduce $\tilde w_h:= w_h \circ \Phi_h^{-1}$ and $\tilde u:= u \circ \Phi_h$ and compare the reference formulation without geometrical errors
\begin{align*}
a^i(u, \tilde w_h) &= \alpha_i \int_{\Omega_i} \nabla u \cdot \nabla \tilde w_h \, dx  = \alpha_i \int_{\Omega_{i,h}} \det(D \Phi_h)  
D \Phi_h^{-T}  \nabla \tilde u \cdot
D \Phi_h^{-T}  \nabla w_h \, dy
\end{align*}
with the discrete variant 
\begin{align*}
a_h^i(\tilde u ,w_h)
& = \alpha_i \int_{\Omega_{i,h}} 
\nabla \tilde u \cdot
\nabla w_h \, dy.
\end{align*}
From this we get
\[
 |a^i(u, \tilde w_h) - a_h^i(\tilde u ,w_h)| = \alpha_i \big| \int_{\Omega_{i,h}} \nabla \tilde u^T \cdot C \nabla w_h \, dy\big| \lesssim \|C\|_{\infty,\Omega} \|\tilde u\|_{H^1(\Omega_{i,h})}\|w_h\|_h,
\]
with $C:= \det(A) (A^TA)^{-1}-I$, $A:=D\Phi_h$. From Lemma ~\ref{lem:phihbounds} we have
$
 \|A-I\|_{\infty,\Omega} \lesssim h^k \text{ and } \|A^{-1}-I\|_{\infty,\Omega} \lesssim h^k, 
$
which implies $\Vert \mathrm{det}(A) - 1 \Vert_{\infty,\Omega} \lesssim h^k$. We thus obtain  the bound
$$
 \Vert C \Vert_{\infty,\Omega} \lesssim \Vert \mathrm{det}(A) - 1 \Vert_{\infty,\Omega} + \Vert A - I \Vert_{\infty,\Omega} + \Vert A^{-1} - I \Vert_{\infty,\Omega} \lesssim h^k.
$$
Combining these estimates results in
\begin{equation} \label{est1} 
 |a^i(u, \tilde w_h) - a_h^i(\tilde u ,w_h)| \lesssim h^k \Vert \tilde u \Vert_{H^1(\Omega_{i,h})} \Vert w_h \Vert_h.
\end{equation}
With similar arguments we  obtain the estimate 
$$\Vert \tilde u \Vert_{H^{1}(\Omega_{i,h})} \sim \Vert u \Vert_{H^{1}(\Omega_i)}.$$

We now consider the terms $N_h^c(\tilde u, w_h)$ and $N^c(\tilde u, w_h)$. 
We use a relation between $ n_{\Gamma_h}(x)=:n_{\Gamma_h}$ and $n_\Gamma(\Phi_h(x))$, $x \in \Gamma_h \cap \thetah(T)=\thetah(\Gammalin_T)$. Let $t_1, \ldots, t_{d-1}$ be an orthonormal basis of $n_{\Gamma_h}^\perp$. The tangent space to $\Gamma= \Phi_h(\Gamma_h)$ at $\Phi_h(x)$ is given by ${\rm span}\{\, D\Phi_h(x)t_j~|~1 \leq j \leq d-1\,\}=n_\Gamma(\Phi_h(x))^\perp$. Hence,
$n_\Gamma(\Phi_h(x)) \in {\rm span}\{\, D\Phi_h(x)^{-T}n_{\Gamma_h}\,\}$ holds. This yields
\begin{equation} \label{traans} n_\Gamma \circ \Phi_h = \frac{D\Phi_h^{-T} n_{\Gamma_h}}{\Vert D\Phi_h^{-T} n_{\Gamma_h} \Vert_2}.
\end{equation}
In the transformation from $\Gamma_h$ to $\Gamma$ we get the factor $\det(D\Phi_h) \Vert D\Phi_h^{-T} n_{\Gamma_h} \Vert_2$ from the change in the measure, cf.~\eqref{changemeas}. The term  $\Vert D\Phi_h^{-T} n_{\Gamma_h} \Vert_2$ is the same as  the denominator in the  change of normal directions \eqref{traans}. Therefore these terms cancel and simplify the formulas below. With the same matrix $C$ as used above, we get 
\begin{align}
  & \big|  N^c(u, \tilde w_h) - N_h^c(\tilde u, w_h) \big| \nonumber \\
 & = \big|\int_{\Gamma_h} \det (D \Phi_h) \, D\Phi_h^{-T} \average{-\alpha \nabla \tilde u } \cdot D\Phi_h^{-T} n_{\Gamma_h} \  \spacejump{w_h} \,  ds -
 \int_{\Gamma_h} \average{-\alpha \nabla \tilde u } \cdot n_{\Gamma_h} \  \spacejump{w_h} \,  ds\big| \nonumber \\
 & = \big| \int_{\Gamma_h} n_{\Gamma_h}^T \cdot C  \average{-\alpha \nabla \tilde u } \spacejump{w_h} \,  ds\big| \lesssim  \|C\|_{\infty,\Omega} \|\average{-\alpha \nabla \tilde u }\|_{-\frac12,h,\Gamma_h} \|\spacejump{w_h}\|_{\frac12,h,\Gamma_h}  \nonumber \\
& \lesssim \|C\|_{\infty,\Omega} 
\|\average{-\alpha \nabla u }\|_{-\frac12,h,\Gamma}
 \|w_h\|_h \lesssim h^{k+\frac12} \|u\|_{H^2(\Omega_1 \cup \Omega_2)}\|w_h\|_h. \label{pl}
\end{align} 
Finally we use $\spacejump{\tilde u} = 0$ at $\Gamma_h$ to note that $N_h^c(w_h, \tilde{u})=N_h^s(\tilde{u},w_h)=N^c(\tilde w_h, u)=N^s(u, \tilde w_h)=0$. Combining this with the bounds in \eqref{est1}, \eqref{pl} and the definition of $A_h(\cdot,\cdot)$ and $A(\cdot,\cdot)$ we obtain the bound in \eqref{p1}.

For the other consistency term in the Strang estimate we obtain:
\begin{align*}
 |f(\tilde w_h)-f_h(w_h)| 
& \leq  \sum_{i=1,2} \left| \int_{\Omega_{i}}  f_i (w_h\circ \Phi_h^{-1}) \, dx - \int_{\Omega_{i,h}} f_{i,h} w_h\, dx \right|,
\end{align*}
and for $i=1,2$ we get, using $f_{i,h}=f_i$ on $\Omega_i$,
\begin{align*}
\hspace*{2cm} & \hspace*{-2cm} \left| \int_{\Omega_{i}}  f_i (w_h\circ \Phi_h^{-1}) \, dx - \int_{\Omega_{i,h}} f_{i,h} w_h\, dx \right| \\
& \leq |\int_{\Omega_{i,h}} (\det(D\Phi_h) (f_{i,h} \circ \Phi_h)  - f_{i,h}) w_h \, dx | \\
& \lesssim \|f_{i,h} \circ \Phi_h - f_{i,h} \|_{\infty,\Omega_{i,h}} \|w_h\|_h + \Vert \det(D\Phi_h) -1 \Vert_{\infty,\Omega_{i,h}} \|f_i\|_{\infty,\Omega_i} \|w_h\|_h.
\end{align*}
Finally note that $ \Vert \det(D\Phi_h) -1 \Vert_{\infty,\Omega} \lesssim h^k$ and $\|f_h\|_{H^{1,\infty}(\Omega_{1,h}\cup \Omega_{2,h})} \lesssim \|f\|_{H^{1,\infty}(\Omega_1 \cup \Omega_2)}$, and thus we get:
\[
   \|f_{i,h} \circ \Phi_h - f_{i,h}\|_{\infty,\Omega_{i,h}} \leq \|f_{i,h}\|_{H^{1,\infty}(\Omega_{i,h})} \|\Phi_h - \id\|_{\infty, \Omega} \lesssim \|f\|_{H^{1,\infty}(\Omega_1\cup\Omega_2)} h^{k+1}
\]
which completes the proof of \eqref{p2}.

\bibliographystyle{siam}
\bibliography{literature}

\begin{thebibliography}{10}

\bibitem{abedian2013performance}
{\sc A.~Abedian, J.~Parvizian, A.~Duester, H.~Khademyzadeh, and E.~Rank}, {\em
  Performance of different integration schemes in facing discontinuities in the
  finite cell method}, International Journal of Computational Methods, 10
  (2013), p.~1350002.

\bibitem{bastian2009unfitted}
{\sc P.~Bastian and C.~Engwer}, {\em An unfitted finite element method using
  discontinuous {Galerkin}}, Int. J. Num. Meth. Eng., 79 (2009),
  pp.~1557--1576.

\bibitem{Basting2013228}
{\sc S.~Basting and M.~Weismann}, {\em A hybrid level set -- front tracking
  finite element approach for fluid--structure interaction and two-phase flow
  applications}, Journal of Computational Physics, 255 (2013), pp.~228 -- 244.

\bibitem{bernardi1989optimal}
{\sc C.~Bernardi}, {\em Optimal finite-element interpolation on curved
  domains}, SIAM J. Numer. Anal., 26 (1989), pp.~1212--1240.

\bibitem{boiveau2016fictitious}
{\sc T.~Boiveau, E.~Burman, S.~Claus, and M.~G. Larson}, {\em Fictitious domain
  method with boundary value correction using penalty-free nitsche method},
  arXiv preprint 1610.04482,  (2016).

\bibitem{burman2014cutfem}
{\sc E.~Burman, S.~Claus, P.~Hansbo, M.~G. Larson, and A.~Massing}, {\em
  {CutFEM}: Discretizing geometry and partial differential equations}, Int. J.
  Num. Meth. Eng.,  (2014).

\bibitem{burman2012fictitious}
{\sc E.~Burman and P.~Hansbo}, {\em Fictitious domain finite element methods
  using cut elements: {II}. a stabilized {Nitsche} method}, Applied Numerical
  Mathematics, 62 (2012), pp.~328--341.

\bibitem{burman2015cut}
{\sc E.~Burman, P.~Hansbo, and M.~G. Larson}, {\em A cut finite element method
  with boundary value correction}, arXiv preprint 1507.03096,  (2015).

\bibitem{cutDG}
{\sc E.~Burman, P.~Hansbo, M.~G. Larson, and A.~Massing}, {\em A cut
  discontinuous galerkin method for the laplace--beltrami operator}, IMA
  Journal of Numerical Analysis,  (2016), p.~drv068.

\bibitem{carraro2015implementation}
{\sc T.~Carraro and S.~Wetterauer}, {\em On the implementation of the extended
  finite element method (xfem) for interface problems}, Archive of Numerical
  Software, 4 (2016).

\bibitem{cheng2010higher}
{\sc K.~W. Cheng and T.-P. Fries}, {\em Higher-order {XFEM} for curved strong
  and weak discontinuities}, Int. J. Num. Meth. Eng., 82 (2010), pp.~564--590.

\bibitem{chernyshenko2015adaptive}
{\sc A.~Y. Chernyshenko and M.~A. Olshanskii}, {\em An adaptive octree finite
  element method for {PDE}s posed on surfaces}, Comput. Meth. Appl. Mech. Eng.,
  291 (2015), pp.~146--172.

\bibitem{CiarletHandbook}
{\sc P.~Ciarlet}, {\em Basic error estimates for elliptic problems}, in
  Handbook of Numerical Analysis, P.~Ciarlet and J.-L. Lions, eds.,
  North-Holland, Amsterdam, 1991, pp.~17--351.

\bibitem{ciarlet1972interpolation}
{\sc P.~G. Ciarlet and P.-A. Raviart}, {\em Interpolation theory over curved
  elements, with applications to finite element methods}, Comput. Meth. Appl.
  Mech. Eng., 1 (1972), pp.~217--249.

\bibitem{deckelnick2014unfitted}
{\sc K.~Deckelnick, C.~M. Elliott, and T.~Ranner}, {\em Unfitted finite element
  methods using bulk meshes for surface partial differential equations}, SIAM
  J. Numer. Anal., 52 (2014), pp.~2137--2162.

\bibitem{demlow2009higher}
{\sc A.~Demlow}, {\em Higher-order finite element methods and pointwise error
  estimates for elliptic problems on surfaces}, SIAM J. Numer. Anal., 47
  (2009), pp.~805--827.

\bibitem{dreau2010studied}
{\sc K.~Dr{\'e}au, N.~Chevaugeon, and N.~Mo{\"e}s}, {\em Studied {X-FEM}
  enrichment to handle material interfaces with higher order finite element},
  Comput. Meth. Appl. Mech. Eng., 199 (2010), pp.~1922--1936.

\bibitem{elliott2012finite}
{\sc C.~M. Elliott and T.~Ranner}, {\em Finite element analysis for a coupled
  bulk--surface partial differential equation}, IMA Journal of Numerical
  Analysis,  (2012), p.~drs022.

\bibitem{engwer2012dune}
{\sc C.~Engwer and F.~Heimann}, {\em Dune-{UDG}: a cut-cell framework for
  unfitted discontinuous {Galerkin} methods}, in Advances in DUNE, Springer,
  2012, pp.~89--100.

\bibitem{fries2010extended}
{\sc T.-P. Fries and T.~Belytschko}, {\em The extended/generalized finite
  element method: an overview of the method and its applications}, Int. J. Num.
  Meth. Eng., 84 (2010), pp.~253--304.

\bibitem{fries2015}
{\sc T.-P. Fries and S.~Omerović}, {\em Higher-order accurate integration of
  implicit geometries}, Int. J. Num. Meth. Eng.,  (2015).

\bibitem{GaldibookNS}
{\sc G.~Galdi}, {\em An Introduction to the Mathematical Theory of the
  Navier-Stokes Equations: Steady-State Problems, 2. Edition}, Springer New
  York, 2011.

\bibitem{gawlik2014high}
{\sc E.~S. Gawlik and A.~J. Lew}, {\em High-order finite element methods for
  moving boundary problems with prescribed boundary evolution}, Comput. Meth.
  Appl. Mech. Eng., 278 (2014), pp.~314--346.

\bibitem{GLR16}
{\sc J.~Grande, C.~Lehrenfeld, and A.~Reusken}, {\em Analysis of a high order
  trace finite element method for pdes on level set surfaces}, arXiv preprint
  1611.01100,  (2016).

\bibitem{grande2014highorder}
{\sc J.~Grande and A.~Reusken}, {\em A higher order finite element method for
  partial differential equations on surfaces}, SIAM J. Numer. Anal., 54 (2016),
  pp.~388--414.

\bibitem{DROPS}
{\sc S.~Gro\ss{} et~al.}, {\em {DROPS} package for simulation of two-phase
  flows}, 2015.

\bibitem{grossreusken07}
{\sc S.~{Gro\ss} and A.~Reusken}, {\em An extended pressure finite element
  space for two-phase incompressible flows}, J. Comput. Phys., 224 (2007),
  pp.~40--58.

\bibitem{ernguermond15}
{\sc J.-L. Guermond and A.~Ern}, {\em Finite element quasi-interpolation and
  best approximation}, ESAIM: Mathematical Modelling and Numerical Analysis,
  (2015).

\bibitem{hansbo2002unfitted}
{\sc A.~Hansbo and P.~Hansbo}, {\em An unfitted finite element method, based on
  {Nitsche’s} method, for elliptic interface problems}, Comput. Meth. Appl.
  Mech. Eng., 191 (2002), pp.~5537--5552.

\bibitem{johanssonhigh2013}
{\sc A.~Johansson and M.~G. Larson}, {\em A high order discontinuous {Galerkin}
  {Nitsche} method for elliptic problems with fictitious boundary}, Numer.
  Math., 123 (2013), pp.~607--628.

\bibitem{karniadakis2013spectral}
{\sc G.~Karniadakis and S.~Sherwin}, {\em Spectral/hp element methods for
  computational fluid dynamics}, Oxford University Press, 2013.

\bibitem{LPWL_PAMM_2016}
{\sc P.~Lederer, C.-M. Pfeiler, C.~Wintersteiger, and C.~Lehrenfeld}, {\em
  Higher order unfitted fem for stokes interface problems}, PAMM, 16 (2016),
  pp.~7--10.

\bibitem{lehrenfeld2015nitsche}
{\sc C.~Lehrenfeld}, {\em The {Nitsche} {XFEM-DG} space-time method and its
  implementation in three space dimensions}, SIAM J. Sci. Comput., 37 (2015),
  pp.~A245--A270.

\bibitem{lehrenfeld2015diss}
{\sc C.~Lehrenfeld}, {\em On a Space-Time Extended Finite Element Method for
  the Solution of a Class of Two-Phase Mass Transport Problems}, PhD thesis,
  RWTH Aachen, February 2015.

\bibitem{lehrenfeld15}
{\sc C.~Lehrenfeld}, {\em High order unfitted finite element methods on level
  set domains using isoparametric mappings}, Comput. Meth. Appl. Mech. Eng.,
  300 (2016), pp.~716--733.

\bibitem{L_ARXIV_2016}
\leavevmode\vrule height 2pt depth -1.6pt width 23pt, {\em A higher order
  isoparametric fictitious domain method for level set domains}, arXiv preprint
  1612.02561,  (2016).

\bibitem{lenoir1986optimal}
{\sc M.~Lenoir}, {\em Optimal isoparametric finite elements and error estimates
  for domains involving curved boundaries}, SIAM J. Numer. Anal., 23 (1986),
  pp.~562--580.

\bibitem{lorensen1987marching}
{\sc W.~E. Lorensen and H.~E. Cline}, {\em Marching cubes: A high resolution 3d
  surface construction algorithm}, in ACM SIGGRAPH Computer Graphics, vol.~21,
  ACM, 1987, pp.~163--169.

\bibitem{massjung12}
{\sc R.~Massjung}, {\em An unfitted discontinuous {Galerkin} method applied to
  elliptic interface problems}, SIAM J. Numer. Anal., 50 (2012),
  pp.~3134--3162.

\bibitem{mayer2009interface}
{\sc U.~M. Mayer, A.~Gerstenberger, and W.~A. Wall}, {\em Interface handling
  for three-dimensional higher-order {XFEM}-computations in fluid--structure
  interaction}, Int. J. Num. Meth. Eng., 79 (2009), pp.~846--869.

\bibitem{muller2013highly}
{\sc B.~M{\"u}ller, F.~Kummer, and M.~Oberlack}, {\em Highly accurate surface
  and volume integration on implicit domains by means of moment-fitting}, Int.
  J. Num. Meth. Eng., 96 (2013), pp.~512--528.

\bibitem{naerland2014geometrychap5}
{\sc T.~A. N{\ae}rland}, {\em Geometry decomposition algorithms for the
  {Nitsche} method on unfitted geometries}, master's thesis, University of
  Oslo, 2014.

\bibitem{Nitsche71}
{\sc J.~Nitsche}, {\em {\"Uber ein Variationsprinzip zur L\"osung von
  Dirichlet-Problemen bei Verwendung von Teilr\"aumen, die keinen
  Randbedingungen unterworfen sind}}, Abh. Math. Sem. Univ. Hamburg, 36 (1971),
  pp.~9--15.

\bibitem{olshanskii2009finite}
{\sc M.~A. Olshanskii, A.~Reusken, and J.~Grande}, {\em A finite element method
  for elliptic equations on surfaces}, SIAM J. Numer. Anal., 47 (2009),
  pp.~3339--3358.

\bibitem{omerovicconformal2016}
{\sc S.~Omerovi{\'c} and T.-P. Fries}, {\em Conformal higher-order remeshing
  schemes for implicitly defined interface problems}, Int. J. Num. Meth. Eng.,
  109 (2017), pp.~763--789.

\bibitem{oswald}
{\sc P.~Oswald}, {\em On a {BPX}-preconditioner for $\mathbb{P}_1$ elements},
  Computing, 51 (1993), pp.~125--133.

\bibitem{parvizianduesterrank07}
{\sc J.~Parvizian, A.~D\"uster, and E.~Rank}, {\em Finite cell method}, Comput.
  Mech., 41 (2007), pp.~121--133.

\bibitem{renard2014getfem++}
{\sc Y.~Renard and J.~Pommier}, {\em {GetFEM++}, an open-source finite element
  library}, 2014.

\bibitem{reusken08}
{\sc A.~Reusken}, {\em Analysis of an extended pressure finite element space
  for two-phase incompressible flows}, Comput. Visual. Sci., 11 (2008),
  pp.~293--305.

\bibitem{saye2015hoquad}
{\sc R.~Saye}, {\em High-order quadrature method for implicitly defined
  surfaces and volumes in hyperrectangles}, SIAM J. Sci. Comput., 37 (2015),
  pp.~A993--A1019.

\bibitem{schoeberl2014cpp11}
{\sc J.~Sch\"oberl}, {\em C++11 implementation of finite elements in
  {NGSolve}}, Tech. Report ASC-2014-30, Institute for Analysis and Scientific
  Computing, September 2014.

\bibitem{sudhakar2013quadrature}
{\sc Y.~Sudhakar and W.~A. Wall}, {\em Quadrature schemes for arbitrary
  convex/concave volumes and integration of weak form in enriched partition of
  unity methods}, Comput. Meth. Appl. Mech. Eng., 258 (2013), pp.~39--54.

\bibitem{warburtonhesthaven03}
{\sc T.~Warburton and J.~Hesthaven}, {\em On the constants in $hp$-finite
  element trace inverse inequalities}, Comput. Meth. Appl. Mech. Eng., 192
  (2003), pp.~2765--2773.

\bibitem{wuxiao10}
{\sc H.~Wu and Y.~Xiao}, {\em An unfitted hp-interface penalty finite element
  method for elliptic interface problems}, arXiv preprint 1007.2893,  (2010).

\end{thebibliography}

\end{document}